\documentclass{amsart}
\makeatletter

\usepackage{amsmath}		
\usepackage{amssymb}		
\usepackage{amsrefs}		    
\usepackage{enumerate}	    
\usepackage[mathcal]{eucal} 
\usepackage{tikz}           
\usepackage{tikz-3dplot}    

\usetikzlibrary{calc,decorations.pathreplacing}

\providecommand{\pc}{\providecommand}
\pc{\nc}{\newcommand}
\pc{\rc}{\renewcommand}

\nc{\bbN}{\mathbb{N}}\nc{\bbZ}{\mathbb{Z}}

\nc{\calC}{\mathcal{C}}\nc{\calD}{\mathcal{D}}
\nc{\calE}{\mathcal{E}}\nc{\calN}{\mathcal{N}}
\nc{\calP}{\mathcal{P}}\nc{\calQ}{\mathcal{Q}}

\rc{\phi}{\varphi}
\pc{\ep} {\varepsilon}

\rc{\subset} {\subseteq}
\rc{\supset} {\supseteq}
\rc{\bar}    {\overline}
\rc{\t}		   {\text}
\pc{\inj}    {\hookrightarrow}
\pc{\surj}   {\twoheadrightarrow}
\pc{\iso}    {\xrightarrow{\sim}}
\nc{\res}    {\upharpoonright}
\nc{\xkappa} {X(\vec\kappa)}
\nc{\xlambda}{X(\vec\lambda)}

\pc{\dmo}{\DeclareMathOperator}
\dmo{\supp}{supp}
\dmo{\cf}  {cf}
\dmo{\Res} {Res}
\dmo{\Mod} {mod}
\dmo{\vmin}{vmin}
\dmo{\id}  {id}

\nc{\declaretheorems}[4]{%
  \newtheorem{#1}[main]{#3}
  \newtheorem{#2}[main]{#4}
  \newtheorem*{#1*}{#3}%
  \newtheorem*{#2*}{#4}%
}

\nc{\declaredefinitions}[4]{%
  \theoremstyle{definition}
  \newtheorem{#1}[main]{#3}
  \newtheorem{#2}[main]{#4}
  \newtheorem*{#1*}{#3}%
  \newtheorem*{#2*}{#4}%
  \theoremstyle{plain}
}

\declaredefinitions{definition}{definitions}{Definition}{Definitions}
\declaredefinitions{remark}{remarks}{Remark}{Remarks}
\declaredefinitions{fact}{facts}{Fact}{Facts}
\declaredefinitions{example}{examples}{Example}{Examples}
\declaretheorems{theorem}{theorems}{Theorem}{Theorems}
\declaretheorems{lemma}{lemmas}{Lemma}{Lemmas}
\declaretheorems{proposition}{propositions}{Proposition}{Propositions}
\declaretheorems{corollary}{corollaries}{Corollary}{Corollaries}
\declaretheorems{question}{questions}{Question}{Questions}

\newtheorem{maintheorem}{Theorem}
\rc{\themaintheorem}{\Alph{maintheorem}}

\newcommand\blfootnote[1]{
    \begingroup
    \renewcommand\thefootnote{}\footnote{#1}
    \addtocounter{footnote}{-1}
    \endgroup
}

\renewenvironment{proof}[1][\proofname]{\par
  \normalfont \topsep6\p@\@plus6\p@\relax
  \trivlist
  \item[\hskip\labelsep\itshape #1\@addpunct{.}]\ignorespaces
}{%
  \unskip\nobreak\hspace{.5em}\qedsymbol\endtrivlist\@endpefalse
}

\makeatother

\title[Computing cohomology of locally profinite sets]{Computational techniques for sheaf cohomology of locally profinite sets}
\author{Mark Schachner}
\address{Cornell University}
\email{mjs623@cornell.edu}
\date{\today}

\begin{document}

\begin{abstract}
  We compute the sheaf cohomology with constant $\bbZ_2$ coefficients of a concrete class of locally profinite sets of independent interest. We introduce $k$-sheer partitions to aid in constructions. It is also shown that questions of intermediate cohomology degrees can be reduced to questions about top cohomology degrees by exhibiting nontrivial top cocycles as pointwise limits of coboundaries.
\end{abstract}

\maketitle

\section{Introduction}
\blfootnote{Supported by NSF grant DGE-2139899.}

\setlength{\parindent}{12pt}

In this paper we describe a toolbox for computing sheaf cohomology groups of certain ``infinitary cuboid" spaces. These are \textit{locally profinite sets}, in that they arise from the deletion of a point from a profinite set. Profinite sets\footnote{More concisely, \textit{Stone spaces}; more verbosely, \textit{totally disconnected compact Hausdorff spaces}.}, in turn, are important for their role in Stone duality, where they form one end of a logical-topological correspondence with Boolean rings.

Recent work has demonstrated that analysis of the cohomology of locally profinite sets, ported across Stone duality, yields novel information about the corresponding rings. Specifically, the construction of a locally profinite set of size $\aleph_n$ and cohomological dimension $n$, due to Aoki \cite{aoki}, yields a nondescendable faithfully flat map of Boolean rings of size $\aleph_\omega$. This, in turn, has implications for the theory of light condensed sets, due to Clausen and Scholze \cites{cond-yt,cond-notes}. Aoki's constructions are infinitary-combinatorial in nature; the essential goal of this paper is to further those arguments by providing novel tools and techniques for computing the cohomology of locally profinite sets. Nevertheless, our methods remain largely elementary, and no forcing machinery is involved.

The paper contains three main parts, corresponding to the three main tools developed. After expositing the spaces under study in section \ref{sec:overview}, in section \ref{sec:main1} we analyze their sheaf cohomology directly by exhibiting nontrivial cocycles via a nonconstructive argument. Here our arguments are within ZFC, but our conclusions are strengthened by cardinal arithmetic bounds. In section \ref{sec:main2}, we define $k$-\textit{sheer} partitions of the spaces, which aid in constructions of cocycles and trivializers. We use these partitions to show both vanishing and nonvanishing of certain cohomology groups; we also prove a converse to one of Aoki's main results. Lastly, in section \ref{sec:main3} we give a sufficient condition for nonvanishing cohomology in terms of nonvanishing top cohomology. All our arguments may be construed as taking place in ZF, with choice only needed in a few places, such as in the existence of an injection $\aleph_1 \inj 2^{\aleph_0}$. 

I am grateful to Justin Tatch Moore and Jeffrey Bergfalk for invaluable advising, as well as to the NSF Graduate Research Fellowship for funding the period of my graduate study during which the majority of this research was undertaken.

\setlength{\parindent}{0pt}

\section{Definitions and overview of results}\label{sec:overview}

Much of our notation is standard:
\begin{itemize}
  \item Tuples of cardinals, ordinals, etc. are denoted by $\vec\kappa, \vec\lambda, \vec\beta, \vec x$, etc. (We reserve $\alpha$ to refer to the one-point compactification $\alpha(X)$ of a space $X$.)
  \item We frequently denote by $n+1$ the set $\{0,\ldots,n\}$, and given $j \in n+1$ and a tuple $\vec x$ we write $\vec x^j$ for the tuple of length $n$ obtained by deleting the $j$-th coordinate from $\vec x$. We also extend this notation as follows: for $A \subset n+1$, the expression $\vec x^A$ refers to the tuple obtained by deleting all elements whose index lies in $A$. By analogy, $\vec x\res_A$ refers to the tuple of elements of $\vec x$ whose indices lie in $A$. 
  \item The expressions $[A]^k, [A]^{\geq k}$ refer to the sets of subsets of $A$ of size $k$ and of size at least $k$, respectively. 
  \item For tuples $\vec \lambda, \vec \kappa$, the relation $\vec\lambda \leq \vec\kappa$ will always refer to the coordinate-wise partial order.
\end{itemize}  
\begin{definition}\label{def:xkappa}
    Let $\vec\kappa$ be a weakly increasing tuple of infinite cardinals. The space $\xkappa$ is defined 
    $$
      \xkappa := \prod_{i \leq n} \alpha(\kappa_i),
    $$
    where as above $\alpha(\kappa_i) = \kappa_i \amalg \{\infty\}$ denotes the one-point compactification of the discrete space of size $\kappa_i$. We denote by $\xkappa^-$ the subspace obtained by deleting the point $(\infty,\ldots,\infty)$.
\end{definition}

Study of the $X(\vec \kappa)^-$ prior to Aoki largely took place in the 1960s; the space $X(\aleph_0,\mathfrak{c})^-$ is an important counterexample in point-set topology, where it's known as \textit{Thomas' plank} \cite{thomasplank}. Wiegand \cites{wiegand1,wiegand2} also considered the spaces $X(\aleph_0,\aleph_1,\ldots,\aleph_n)^-$, which Aoki showed have nonzero cohomology in the $n$-th degree.

\begin{remark}\label{rem:xkappa-omega}
  Our notation differs slightly from Aoki's; for instance, where he writes $X(0,1,2)$ we write $X(\aleph_0,\aleph_1,\aleph_2)$. We will still typically consider the case where $\sup\vec\kappa < \aleph_\omega$.
\end{remark}

A Leray cover for $\xkappa^-$, i.e., one in which all finite intersections are acyclic, is given by the sets $\{\vec x : x_i \neq \infty\}$ for each $i \leq n$. Aoki used this cover to characterize $H^\bullet(\xkappa^-,\underline{\bbZ_2})$ via the associated \v{C}ech complex. We introduce it via the following notation:

\pagebreak
\begin{definition}\label{def:combchar-notation}
    Let $A \subset n+1$, and define 
    \begin{align*}
      \calD_A(\vec{\kappa}) &= \{\langle x_i\rangle \in \xkappa : \forall i \in A, x_i \neq \infty\}, \\
      \calC_A(\vec\kappa) &= \{f : \calD_A(\vec\kappa) \to \bbZ_2 \t{ continuous}\},
    \end{align*}
    where
    \begin{itemize}
      \item $\calD_A{(\vec\kappa)}$ is given the subspace topology;
      \item $\bbZ_2$ is given the discrete topology; and
      \item $\calC_A{(\vec\kappa)}$ inherits the ring structure from $\bbZ_2$,
      and is given the compact-open topology.
    \end{itemize}
    We omit the tuple $\vec\kappa$ from the notation when clear, writing $\calD_A,\calC_A$.
  \end{definition}

\begin{proposition}[\cite{aoki}]\label{prop:combchar}
The sheaf cohomology of $\xkappa^-$ with constant $\bbZ_2$ coefficients is computed
by the \v{C}ech complex whose $k$-th term is 
$$
    \check{C}^k(\xkappa^-,\underline{\bbZ_2}) = \prod_{A \in [n+1]^{k+1}} \calC_A(\vec\kappa).
$$
\end{proposition}
More concretely:
  \begin{itemize}
    \item a $k$-cochain is (represented by) a tuple of functions $\vec f = \langle f_A : A \in [n+1]^{k+1}\rangle,$
    where each $f_A : \calD_A(\vec \kappa) \to \bbZ_2$ is continuous; i.e., $f_A$ 
    is an ordinary function whose support is a clopen subset of $\calD_A$.
    \item A tuple $\vec f$ represents a $k$-cocycle if, for any $A' \in [n+1]^{k+2}$, we have 
    $$
      (d\vec f)_{A'} := \sum_{i \in A'} f_{A' \setminus \{i\}} = 0.
    $$
    \item A $k$-coboundary arises from some tuple $\langle g_{A'} : A' \in [n+1]^k\rangle$ via 
    $$
      f_A = \sum_{i \in A} g_{A \setminus \{i\}}.
    $$
\end{itemize}

Aoki notes that Proposition \ref{prop:combchar} immediately implies that the 
cohomological dimension of $\xkappa^-$ is at most $n$; moreover, he uses this 
characterization to show that $H^n(\xkappa^-,\underline{\bbZ_2}) \neq 0$ whenever $\kappa_i \geq \aleph_i$ for all $i \leq n$.
Here, a top cocycle is just a function 
$f : \prod_{i \leq n} \kappa_i \to \bbZ_2$, and $f$ represents a 
coboundary if and only if it is the sum of maps
$f_i : \prod_{i \leq n} \kappa_i \to \bbZ_2, i \leq n$, such that
$f_i$ is mod-finite independent of the $i$-th coordinate.

Our first main tool, proved using this characterization, is the following.

\begin{maintheorem}\label{thm:main1}
  Let $0 < k < n \in \bbN$, and suppose 
  $\vec\kappa = \langle \kappa_0 \leq \kappa_1 \leq \cdots \leq \kappa_n \rangle$
  is a weakly increasing tuple of infinite cardinals such that
  $$
    \kappa_{n-k} < \cf \kappa_j \t{ for every } j > n-k
    \quad \t{ and } \quad 
    2^{\kappa_0} \geq \kappa_{n-k+1},
  $$
  then $H^k(\xkappa^-,\underline{\bbZ_2}) \neq 0$.
\end{maintheorem}

As a simple corollary, we obtain ZFC ``reference spaces" for nonvanishing $H^k$ among the $\xkappa^-$, namely those indexed by tuples of the form
$$
\vec\kappa = \langle\underbrace{\aleph_0,\aleph_0,\ldots,\aleph_0}_{n-k+1}, \underbrace{\aleph_1,\aleph_1,\ldots,\aleph_1}_{k}\rangle.
$$
\newpage
In section \ref{sec:main2}, we describe techniques for eliminating the cardinal arithmetic assumptions of Theorem \ref{thm:main1} in certain situations. The main tool is a partition of $\xkappa^-$ into $n+1$ pieces, each of which is ``thin'' in one of the coordinate directions; we call such partitions \textit{sheer}.

We use this machinery to obtain both vanishing and nonvanishing results:

\begin{maintheorem}\label{thm:main2}
  Let $\vec\kappa$ be a length-$(n+1)$ tuple of infinite cardinals which admits a $k$-sheer partition. Then:
  \begin{enumerate}
    \item $H^{k+1}(\xkappa^-,\underline{\bbZ_2}) = 0,$ and
    \item For any $\vec\lambda \leq \vec\kappa$ coordinatewise, the restriction map
    $$
     H^k(\xkappa^-,\underline{\bbZ_2}) \to H^k(\xlambda^-,\underline{\bbZ_2})
    $$
    is surjective. In particular, if $H^k(\xlambda^-,\underline{\bbZ_2}) \neq 0$ then $H^k(\xkappa^-,\underline{\bbZ_2}) \neq 0.$
  \end{enumerate}
\end{maintheorem}

We determine exactly when $k$-sheer partitions exist, in the process deriving a converse to one of Aoki's results:
\begin{corollary}\label{thm:aoki-converse}
Let $\vec\kappa$ be a tuple of infinite cardinals such that $\kappa_i < \aleph_i$ for some $i \leq n$. Then $H^n(\xkappa^-,\underline{\bbZ_2}) = 0$.
\end{corollary}

Lastly, in the final section we prove the following sufficient condition for nonvanishing cohomology.

\begin{maintheorem}\label{thm:main3}
  Suppose there exists $B \in [n+1]^{k+1}$ and a continuous map \linebreak $\psi : X(\vec\kappa^B)^- \to \check{C}^{k-1}(X(\vec\kappa\res B)^-,\underline{\bbZ_2})$ such that the map $\vec x \mapsto d\psi(\vec x): X(\vec\kappa^B)^- \to Z^{k}(X(\vec\kappa\res B)^-,\underline{\bbZ_2})$ extends to a continuous map $\phi: X(\vec\kappa^B) \to Z^k(X(\vec\kappa\res_B)^-, \underline{\bbZ_2})$, and $\phi(\infty,\ldots,\infty)$ is not a coboundary.
  Then $H^k(\xkappa^-,\underline{\bbZ_2}) \neq 0$.
\end{maintheorem}

Theorem \ref{thm:main3} can be read in the following manner: let $\vec f = \phi(\infty,\ldots,\infty)$. Then $\vec f$ is nonzero in $H^k(\vec \kappa\res_B)$, but the map $\psi$ exhibits $\vec f$ as a limit of coboundaries. Theorem \ref{thm:main3} makes use of this incompactness property to construct a nonzero cocycle in the intermediate cohomology of the larger space $\xkappa^-$.

\section{Proof of Theorem \ref{thm:main1}}\label{sec:main1}

We begin with an analysis of the cardinality and relationships between the 
$\calC_A(\vec\kappa)$, for varying $A$ and $\vec\kappa$. The following two 
facts will be useful; their proofs are routine.

\begin{facts*}\label{facts} \hfill
  \begin{enumerate}
    \item If $A \subset A'$, then $\calD_{A'} \subset \calD_A$ is open dense, so that the map 
    $f \mapsto f\restriction\calD_{A'}$ defines a dense embedding of topological rings
    $\calC_A \inj \calC_{A'}$. In light of this and related observations, we will frequently abuse notation when dealing with functions defined on different domains; for instance, given functions $f_A \in \calC_A$ and $f_B \in \calC_B$, we write $f_A + f_B$ for the function defined on the intersection $\calD_A \cap \calD_B = \calD_{A \cup B}$.
    \item If $j \in A$, then a continuous map $f \in \calC_A(\vec\kappa)$
    corresponds uniquely to a $\kappa_j$-tuple of maps 
    $f_\beta \in \calC_{A \setminus\{j\}}(\vec\kappa^j)$, and this correspondence
    in fact defines an isomorphism of topological rings
    $$
      \calC_A{(\vec\kappa)} \cong \left(\calC_{A \setminus \{j\}}(\vec\kappa^j)\right)^{\kappa_j}
    $$
    More generally, any inclusion $A \subset A'$ yields an isomorphism of topological rings
    $$
      \calC_{A'}{(\vec\kappa)} \cong 
      \calC_{A' \setminus A}(\vec\kappa^A)^{\prod\limits_{j \in A}\kappa_j}.
    $$
  \end{enumerate}
\end{facts*}

As a first application of these observations, we find the cardinality of 
$\calC_A(\vec\kappa)$ for arbitrary $A,\vec\kappa$.
\begin{proposition}\label{prop:cardC}
  Suppose $\vec\kappa = \langle \kappa_0 \leq \ldots \leq \kappa_n\rangle$ is a 
  weakly increasing $(n+1)$-tuple of infinite cardinals, and let $A \subset n+1$. Then:
  \begin{enumerate}[(i)]
    \item if $A = \{0,\ldots,n\}$ then $|\calC_A(\vec\kappa)| = 2^{\kappa_n}$;
    \item if $A = \emptyset$ then $|\calC_A(\vec\kappa)| = \kappa_n$; and
    \item if $A$ and its complement $A^c$ are nonempty then 
    $$
      |\calC_A(\vec\kappa)| = \left(\kappa_{\max A^c}\right)^{\kappa_{\max A}}.
    $$
  \end{enumerate}
\end{proposition}

\begin{proof}
  Item (i) follows because $\calD_{\{0,\ldots,n\}}$ is just the discrete topology
  on $\prod_{i \leq n} \kappa_i$. To show (ii), we first note $\xkappa$ has $\kappa_n$ many isolated points; the indicator functions for each of these are continuous, so that $|\calC_\emptyset(\vec\kappa)| \geq \kappa_n$. To show $|\calC_\emptyset(\vec\kappa)| \leq \kappa_n$, we induct on the length of $\vec \kappa$:
  \begin{itemize}
    \item If $n = 0$ then $\calC_\emptyset(\vec\kappa)$ is just the ring of 
    continuous functions $\alpha(\kappa_0) \to \bbZ_2$. These correspond to 
    the finite-or-cofinite subsets of $\kappa_0$, of which there are 
    $\kappa_0$ many.
    \item Now suppose $|\calC_\emptyset(\vec\kappa)| \leq \kappa_n$
    for any tuple $\vec\kappa$ of length at most $n$, and fix a continuous function 
    $f : \prod_{i \leq n} \alpha(\kappa_i) \to \bbZ_2$. Since the codomain 
    of $f$ is discrete, the preimage under $f$ of $f(\infty,\ldots,\infty)$ is a neighborhood
     of $(\infty,\ldots,\infty)$. Such a neighborhood contains some rectangle 
     $R = \prod_{i \leq n} J_i$, where $J_i \subset \alpha(\kappa_i)$ is 
     cofinite and contains $\infty$. We can then write the domain of $f$ as the 
     disjoint union of finitely many subspaces, as follows:
    \begin{align*}
      \prod_{i \leq n} \alpha(\kappa_i)
      &\cong \prod_{i \leq n} \left( J_i \amalg (\alpha(\kappa_i) \setminus J_i)\right) \\
      &\cong \coprod_{A \subset n+1} \left(\prod_{i \in A} J_i \times \prod_{i \not \in A} (\alpha(\kappa_i) \setminus J_i)\right).
    \end{align*}
    This is a disjoint union of $R$ and finitely many spaces of the form 
    $\prod_{i \in A} J_i$, where $A \subset n+1$ is of size $<n+1$. 
    Since $J_i$ is a cofinite subset of $\alpha(\kappa_i)$ containing $\infty$, 
    any space of this form is finite or homeomorphic to $X(\vec\kappa\restriction_A)$, 
    so by the inductive hypothesis the function ring has cardinality at most $\kappa_{\max A}$.

    We have thus described an injection 
    $$
      \calC_\emptyset(\vec \kappa) \inj 
      \prod_{i \leq n}[\kappa_i]^{<\aleph_0} \times \{0,1\} \times \prod_{\emptyset \neq A \subsetneq n+1} \kappa_{\max A}^{<\aleph_0} \times \{0,1\}^{<\aleph_0},
    $$
    where the first coordinates record the rectangle $R$ on which $f$ is constant, 
    the second records this constant value, and the last coordinates record the 
    restrictions of $f$ to each subspace along with the value of $f$ at each isolated point. The 
    latter product has cardinality at most $\kappa_n$, so 
    $|\calC_\emptyset(\vec \kappa)| \leq \kappa_n$ as desired to show (ii).
  \end{itemize}
    To conclude, we now need only note that by Fact 2 above we have
    $$
      \calC_A(\vec\kappa) \cong \calC_\emptyset(\vec\kappa^A)^{\prod\limits_{j \in A}\kappa_j}
    $$
    so that (iii) follows from (ii).
\end{proof}

\
We will also need the following lemma; note that its statement leverages the abuse of notation from Fact 1.

\begin{lemma}\label{lemma:intersectC}
  For every tuple $\vec\kappa$ of length $n+1$, 
  $$
    \left|\bigcap_{i \leq n} \calC_{\{i\}}(\vec\kappa)\right| \geq 2^{\kappa_0}.
  $$
\end{lemma}

\begin{proof}
  Fix $B \subset \kappa_0$. We define 
  \begin{align*}
    \Delta B &: \kappa_0 \times \cdots \times \kappa_n \to \bbZ_2,\\
    \Delta B(\vec x) &= 
    \begin{cases}
      1 &\t{if $x_i = x_j \in B$ for all $i,j$,} \\
      0 &\t{otherwise.}
    \end{cases}
  \end{align*}

  To see $\Delta B \in \calC_{\{i\}}$ for all $i$ it suffices to show it is constant 
  in a neighborhood of any limit point in $\xkappa^-$. So fix $\vec x \in \xkappa^-$, 
  and let $x_j$ be a coordinate which is not $\infty$. If 
  $\vec x$ is the constant tuple $\langle x_j,\ldots,x_j\rangle$, then $\vec x$ is isolated, since 
  $x_j \neq \infty$; so certainly $\Delta B$ is constant on a neighborhood of $\vec x$. 
  Otherwise, $\vec x$ lies in the open ``punctured hyperplane'' 
  $$
    \{\vec y \in \xkappa^- : x_j = y_j\} \setminus \{\langle x_j,\ldots,x_j\rangle\},
  $$
  on which $\Delta B$ is constant.\footnote{Note that if $x_j \not \in \kappa_0$, the removal of the latter point is vacuous.}

  Lastly notice that $\Delta: 2^{\kappa_0} \to \bigcap_{i \leq n}\calC_{\{i\}}(\vec\kappa)$ is certainly injective, since if $x \in B \triangle B'$ then $\Delta B(x,\ldots,x) \neq \Delta B'(x,\ldots,x)$.
\end{proof}

We are now ready to prove Theorem \ref{thm:main1}. We will benefit from the following notation: for a fixed tuple $\vec\kappa$ of length $n+1$ and fixed $k \leq n$, write 
  $$
    \calP = \prod_{i=0}^{n-k} \kappa_i, \qquad \calQ = \prod_{i=n-k+1}^n \kappa_i.
  $$
We equip $\calP$ and $\calQ$ with the coordinatewise partial order, so that 
$|\calP| = \kappa_{n-k}$ and the unbounding number $\t{add}(\calQ)$ is $\min\limits_{j > n-k} \cf{\kappa_j}$.

We also make the following definition for convenience:
\begin{definition}\label{def:k-terminal}
  A subset $A \subset n+1$ is $k$-\textit{terminal} if it contains the last $k$ 
  elements of $n+1$, i.e., $\{n-k+1,\ldots,n\} \subset A$.
\end{definition}

\begin{proposition}\label{prop:files-cofinal}
  Suppose $\kappa_{n-k} < \min\limits_{j > n-k} \cf{\kappa_j}$, and let 
  $\vec f = \langle f_A\rangle$ be a $(k-1)$-cochain. Then for any $k$-terminal 
  $A \in [n+1]^{k+1}$ there is a cofinal subset $C \subset \calQ$ such that the map 
  \begin{align*}
    C  &\to \calC_\emptyset(\kappa_0,\ldots,\kappa_{n-k}) \\
    \vec y &\mapsto (d\vec f)_{A}(-,\vec y)
  \end{align*}
  is constant. 
\end{proposition}

\newpage
\begin{proof}
  Expanding the definition of $d\vec f$ using that $A$ is $k$-terminal,
  \begin{align*}
    (d\vec f)_A = f_{\{n-k+1,\ldots,n\}} + \sum_{j = n-k+1}^n f_{A \setminus \{j\}}.
  \end{align*}
  The codomain of $f_{\{n-k+1,\ldots,n\}}(-,\vec y),$ i.e., 
  $\calC_\emptyset(\kappa_0,\ldots,\kappa_{n-k}),$ has cardinality 
  $\kappa_{n-k} < \t{add}(\calQ)$, so the first summand is constant on a cofinal
   subset $C_0$ of $\calQ$.
  
  For the remaining summands, fix $j \in \{n-k+1,\ldots,n\}$. 
  Since $f_{A \setminus \{j\}} \in \calC_{A \setminus \{j\}}$, by Fact 2 we have that for any $\vec x \in \calP$
  $$
    f_{A \setminus\{j\}}(\vec x, -) \in \calC_{\{n-k+1,\ldots,\hat{j},\ldots,n\}}(\kappa_{n-k+1},\ldots,\kappa_n).
  $$
  That is, if $\vec y \in X(\kappa_{n-k+1},\ldots,\kappa_n)$ satisfies 
  $y_i = \infty$ iff $i = j$, then $f_{A \setminus \{j\}}(\vec x,-)$ must be constant in a 
  neighborhood of $\vec y$. Thus there is $\beta(\vec x, j) \in \kappa_j$ 
  such that $f_{A \setminus \{j\}}(\vec x,\vec y)$ is independent of $y_j$ whenever $y_j > \beta(\vec x, j)$. 
  Taking 
  $$
    \vec\beta(\vec x) = \langle \beta(\vec x, j): n-k+1 \leq j \leq n\rangle \in \calQ,
  $$
  we find that $\sum_{j = n-k+1}^n f_{A \setminus \{j\}}(\vec x, \vec y)$ is independent of $\vec y$ 
  whenever $\vec y > \vec\beta(\vec x)$ coordinatewise.

  The above describes a map $\vec x \mapsto \vec \beta(\vec x): \calP \to \calQ$; so 
  again by our assumption the image of this map must be bounded in $\calQ$. Taking 
  $\vec \beta$ to be a common upper bound in $\calQ$ of all the $\vec \beta(\vec x)$, 
  we find that $\sum_{j = n-k+1}^n f_{A \setminus \{j\}}(\vec x, \vec y)$ is constant for $\vec y > \vec \beta$.

  To conclude, we let $C = \{\vec y \in C_0 : \vec y > \vec \beta\}$. Then $C$ 
  remains cofinal in $\calQ$, and the above arguments and the expanded definition of 
  $d\vec f$ show that $(d \vec f)_A(-,\vec y)$ is constant on $C$.
\end{proof}

Proposition \ref{prop:files-cofinal} describes a condition satisfied by any
$k$-coboundary. As the final step, we construct a $k$-cocycle which does not 
satisfy this condition.

\begin{proposition}\label{prop:nontrivial-cocycle}
  Suppose $2^{\kappa_0} \geq \kappa_{n-k+1}$. Then there is a $k$-cocycle $\vec f = \langle f_A\rangle$ such that for any $k$-terminal $A$ the map $\vec y \mapsto f_A(-,\vec y)$ is not constant on any cofinal subset of $\calQ$. In particular, if $\kappa_{n-k} < \min\limits_{j > n-k} \cf{\kappa_j}$, then $\vec f$ represents a nonzero cohomology class in $H^k(\xkappa^-,\underline{\bbZ_2})$.
\end{proposition}

\begin{proof}
  Starting from Lemma \ref{lemma:intersectC}, we note that 
  \begin{align*}
    \left| \bigcap_{i \leq n-k} \calC_{\{i\}}(\kappa_0,\ldots,\kappa_{n-k}) \right|
    &\geq 2^{\kappa_0}
  \end{align*}
  and the latter is at least $\kappa_{n-k+1}$ by our hypothesis. 
  So fix an injection 
  $$
    \sigma : \kappa_{n-k+1} \inj \bigcap_{i \leq n-k} \calC_{\{i\}}(\kappa_0,\ldots,\kappa_{n-k}),
  $$
  and for $A \in [n+1]^{k+1}$ define 
  \begin{align*}
    f_A &: \calD_A(\vec\kappa) \to \bbZ_2, \\
    f_A(\vec x) &:= 
    \begin{cases}
      \sigma(x_{n-k+1})(x_0,\ldots,x_{n-k}) \quad &\t{if $A$ is $k$-terminal,} \\
      0\quad &\t{otherwise.}
    \end{cases}
  \end{align*}
  Notice first that each $f_A$ is well-defined, since if 
  $A = \{i\} \cup \{n-k+1,\ldots,n\}$ and $\vec x \in \calD_A(\vec\kappa)$ then 
  $x_i,x_{n-k+1} \neq \infty$. 
  
  To see that each $f_A$ is continuous, assume $A$ is $k$-terminal; otherwise 
  $f_A$ is identically zero and thus certainly continuous. Recalling the isomorphism from Fact 2 
  \begin{align*}
    \calC_{\{i\}}(\kappa_0,\ldots,\kappa_{n-k})^\calQ &\cong \calC_{\{i\} \cup \{n-k+1,\ldots,n\}}(\vec\kappa) \\
    &= \calC_A(\vec\kappa),
  \end{align*}
  we see that $f_A$ arises from the function 
  $$
    \langle x_{n-k+1},\ldots,x_n \rangle \mapsto \sigma(x_{n-k+1})
  $$
  and thus lies in $\calC_A(\vec\kappa)$.
  
  To see $\vec f$ defines a cocycle, choose any $A' \in [n+1]^{k+2}$. 
  Then by definition
  \begin{align*}
    (d\vec f)_{A'} = \sum_{i \in A'} f_{A' \setminus \{i\}}.
  \end{align*}
  We will show this sum always vanishes. Again, the case where $A'$ is not 
  $k$-terminal is immediate since every term is zero. So suppose 
  $A' = \{i_0,i_1\} \cup \{n-k+1,\ldots,n\}$. Then the only nonzero summands are 
  those indexed by $i=i_0,i_1$, and $d\vec f$ is independent of these coordinates; so
  \begin{align*}
    (df)_{A'}(\vec x) &= f_{A' \setminus \{i_0\}}(\vec x) + f_{A' \setminus \{i_1\}}(\vec x) \\
    &=0.
  \end{align*}
  So $\vec f$ is a cocycle, as desired.
  
  We conclude by noting that, for any $k$-terminal $A \in [n+1]^{k+1}$, the map 
  \begin{align*}
    \calQ &\to \calC_\emptyset(\kappa_0,\ldots,\kappa_{n-k}) \\
    \vec x &\mapsto f_A(-,\vec x) \\
           &\equiv \sigma(x_{n-k+1})
  \end{align*}
  cannot be constant on a cofinal subset of $\calQ$; any such subset certainly 
  contains tuples whose first coordinates disagree, and the injectivity of 
  $\sigma$ guarantees that the $f_A(-,\vec x)$ will be distinct for tuples with 
  differing first coordinate. Hence $\vec f$ is nontrivial, as desired.
\end{proof}

\section{Constructions via sheer partitions}\label{sec:main2}

As mentioned above, our second main computational tool is the notion of a $k$-sheer partition of $\xkappa^-$. To define this, we begin with a more careful dissection of the topology of $\xkappa$. We make the following notation:

\begin{definition}\label{def:basis}
  For any $A \subset n+1$, write $F_A$ for the 
$A$-\textit{facet} of $\xkappa$, i.e.,
$$
  F_A = \{\vec x \in \xkappa : x_i \neq \infty \t{ if and only if } i \in A\}.
$$
Note that $\xkappa$ is the disjoint union of the facets $F_A$ for $A \subset n+1$, and
for any $A \subset n+1$ we have
$$
  \calD_A(\vec\kappa) = \bigcup_{A \subset B \subset n+1} F_B.
$$

\newpage
A basis for the topology of $\xkappa$ can then be given as follows: for any $A \subset n+1$,
any $\vec x \in F_A$, and any tuple of finite sets
$$
  \vec \calE \in \prod_{j \not \in A} [\kappa_j]^{<\aleph_0},
$$
we define the open set
$$
  \calN(\vec x, \vec \calE) := \{\vec y \in \xkappa : \vec y \res_A = \vec x \res_A 
  \t{ and } y_j \not \in \calE_j \t{ for all } j \not \in A\}.
$$
\end{definition}

That is, $\vec y \neq \vec x$ lies in this neighborhood if and only if it agrees with $\vec x$ wherever $x_j$ is not $\infty$, and all other coordinates are either $\infty$ or lie outside the corresponding finite set $\calE_j$. Figure \ref{fig:cuboid} illustrates the sets $F_A$ and $\calN(\vec x, \vec \calE)$
for a particular choice of $\vec x$ and $\vec \calE$.

\begin{figure}[htbp]
  \centering

  \tdplotsetmaincoords{70}{110}
  \begin{tikzpicture}[scale=2.5,tdplot_main_coords]
  \coordinate (O) at (0,0,0);
  \tdplotsetcoord{P}{3.2}{75}{65}

  \draw[fill=pink,opacity=1,opacity=0.5,scale=0.3,xshift=-5pt,yshift=-73pt] 
  (0,0.1,0.5)  .. controls (0,0.5,0.8)
  .. (1,1.5,1.1)      .. controls (-1,1,0.3) and (-1,1.3,0.3)
  .. (1,2.3,1.1)      .. controls (1.5,3,1.5) and (1.5,3.8,1)
  .. (1,2.6,0.1)      .. controls (1,2,0)
  .. (1,1,0.1)      .. controls (1,1,0.1) and (1,0.2,0)
  .. (1,0.1,0.25)      .. controls (1,0,0.5)
  .. cycle;

  \draw[fill=pink,opacity=1,opacity=0.5] 
    (0,0.1,0.5)  .. controls (0,0.5,0.8)
  .. (1,1.5,1.1)      .. controls (-1,1,0.3) and (-1,1.3,0.3)
  .. (1,2.3,1.1)      .. controls (1.5,3,1.5) and (1.5,3.8,1)
  .. (1,2.6,0.1)      .. controls (1,2,0)
  .. (1,1,0.1)     node[color=black,opacity=1,pos=1,inner sep=2em,anchor=north]{Interior $=F_{012}$} .. controls (1,1,0.1) and (1,0.2,0)
  .. (1,0.1,0.25)      .. controls (1,0,0.5)
  .. cycle;

  \fill [cyan,opacity=0.25](Pxy) -- (Py) -- (Pyz) node[blue,pos=0.3,anchor=east,opacity=1]{$F_{01}$} -- (P);
  \fill [cyan,opacity=0.25](Pxy) -- (Px) node[blue,pos=0.8,anchor=south,inner sep = 1em,opacity=1]{$F_{02}$} -- (Pxz)  -- (P);
  \fill [cyan,opacity=0.25](Pyz) -- (Pz) node[blue,pos=0.7,anchor=north,opacity=1,]{$F_{12}$} -- (Pxz)  -- (P);

  \draw[thick,->] (O) -- (0,0,1) node[anchor=south]{$0$};
  \draw[thick,->] (O) -- (2,0,0) node[anchor=north east]{$1$};
  \draw[thick,->] (O) -- (0,3,0) node[anchor=north west]{$2$};
  \draw[thick] (O) -- (Px);
  \draw[thick] (O) -- (Py);
  \draw[thick] (O) -- (Pz);
  \draw[thick] (Px) -- (Pxy);
  \draw[thick] (Py) -- (Pxy);
  \draw[thick] (Px) -- (Pxz);
  \draw[thick] (Pz) -- (Pxz);
  \draw[thick] (Py) -- (Pyz);
  \draw[thick] (Pz) -- (Pyz);
  \draw[ultra thick, color=orange] (Pxy) -- (P) node[pos=0.1,anchor=south east]{$F_0$};
  \draw[ultra thick, color=orange] (Pyz) -- (P) node[pos=0.3,anchor=east,xshift=-2pt]{$F_1$};
  \draw[ultra thick, color=orange] (Pxz) -- (P) node[pos=0.175,anchor=south east]{$F_2$};

  \filldraw[fill=white] (P) circle (1pt);
  \end{tikzpicture}

  \begin{tikzpicture}[scale=2.5,tdplot_main_coords]
    \coordinate (O) at (0,0,0);
    \tdplotsetcoord{P}{3.2}{75}{65}

    \draw[thick,->] (0,0,0) -- (0,0,1) node[anchor=south]{$0$};
    \draw[thick,->] (0,0,0) -- (2,0,0) node[anchor=north east]{$1$};
    \draw[thick,->] (0,0,0) -- (0,3,0) node[anchor=north west]{$2$};
    \draw[thick] (O) -- (Px);
    \draw[thick] (O) -- (Py);
    \draw[thick] (O) -- (Pz);
    \draw[thick] (Px) -- (Pxy);
    \draw[thick] (Py) -- (Pxy);

    \coordinate (X) at ($ (Pyz)!0.5!(P) $);
    \coordinate (Xa) at ($ (Py)!0.5!(Pxy) $);
    \coordinate (Xb) at ($ (O)!0.5!(Px) $);
    \coordinate (Xc) at ($ (Pz)!0.5!(Pxz) $);

    \draw[thick] 
      ($ (Xa) + (0,0,-0.05) $) -- ($ (Xa) + (0,0,0.05) $);
    \draw[thick] 
      ($ (Xb) + (0,0,-0.05) $) -- ($ (Xb) + (0,0,0.05) $);
    \draw[thick] 
      ($ (Xc) + (0,0,-0.05) $) -- ($ (Xc) + (0,0,0.05) $);
    \draw[dashed,thick] (Xa) -- (X);
    \draw[dashed,thick] (Xa) -- (Xb);
    \draw[dashed,thick] (Xb) -- (Xc);
    \draw[dashed,thick] (Xc) -- (X);

    \fill[yellow,opacity=0.5] 
      (Xb) -- ($ (Xb) + (0,0,0.2) $) -- 
      ($ (Xb) + (0,0.3,0.2) $) -- ($ (Xb) + (0,0.3,0) $);
    
    \fill[yellow,opacity=0.5] 
      ($ (Xb) + (0,0,0.833) $) -- ($ (Xb) + (0,0.3,0.833) $) --
      ($ (Xb) + (0,0.3,0.5) $) -- ($ (Xb) + (0,0,0.5) $);

    \fill[yellow,opacity=0.5] 
      ($ (Xb) + (0,0.9,0) $) -- ($ (Xb) + (0,1.5,0) $) --
      ($ (Xb) + (0,1.5,0.2) $) -- ($ (Xb) + (0,0.9,0.2) $);

    \fill[yellow,opacity=0.5] 
      ($ (Xb) + (0,0.9,0.5) $) -- ($ (Xb) + (0,1.5,0.5) $) --
      ($ (Xb) + (0,1.5,0.833) $) -- ($ (Xb) + (0,0.9,0.833) $);

    \fill[yellow,opacity=0.5] 
      ($ (Xb) + (0,2,0) $) -- ($ (Xb) + (0,2.8,0) $) --
      ($ (Xb) + (0,2.8,0.2) $) -- ($ (Xb) + (0,2,0.2) $);

    \fill[yellow,opacity=0.5] 
      ($ (Xb) + (0,2,0.5) $) -- ($ (Xb) + (0,2.8,0.5) $) --
      ($ (Xb) + (0,2.8,0.833) $) -- ($ (Xb) + (0,2,0.833) $);

    \draw[thick] (Px) -- (Pxz);
    \draw[thick] (Pz) -- (Pxz);
    \draw[thick] (Py) -- (Pyz);
    \draw[thick] (Pz) -- (Pyz);
    \draw[thick] (Pxy) -- (P);
    \draw[thick] (Pyz) -- (P);
    \draw[thick] (Pxz) -- (P);

    \draw[ultra thick,color=red] (1.3,0,0.2) -- (1.3,0,0.5);
    \draw[ultra thick,color=red,decorate,decoration ={
      brace,amplitude=5pt,raise=3pt
    }] (1.3,0,0.2) -- (1.3,0,0.5)
      node[midway,anchor=east,xshift=-7pt]{$\calE_0$};

    \draw[dashed, thick,color=red] (1.3,0,0.2) -- ($ (Xb) + (0,0,0.2) $);
    \draw[dashed, thick,color=red] (1.3,0,0.5) -- ($ (Xb) + (0,0,0.5) $);
    \draw[dashed, thick,color=red] 
      ($ (Xb) + (0,0,0.2) $) -- ($ (Xa) + (0,0,0.2) $);
    \draw[dashed, thick,color=red] 
      ($ (Xb) + (0,0,0.5) $) -- ($ (Xa) + (0,0,0.5) $);
  
    \draw[ultra thick,color=blue] (1.3,0.3,0) -- (1.3,0.9,0);
    \draw[ultra thick,color=blue,decorate,decoration={
      brace,amplitude=5pt,raise=2.5pt,mirror
    }] (1.3,0.3,0) -- (1.3,0.9,0);
    \draw[ultra thick,color=blue,decorate,decoration={
      brace,amplitude=5pt,raise=10pt,mirror
    }] (1.3,0.6,0) -- (1.3,1.75,0)
      node[midway,yshift=-23pt,xshift=-1pt]{$\calE_2$};
    \draw[dashed, thick,color=blue] (1.3,0.3,0) -- ($ (Xb) + (0,0.3,0) $);
    \draw[dashed, thick,color=blue] (1.3,0.9,0) -- ($ (Xb) + (0,0.9,0) $);
    \draw[dashed, thick,color=blue] 
      ($ (Xb) + (0,0.3,0) $) -- ($ (Xa) + (0,-2.5,0.833) $);
    \draw[dashed, thick,color=blue] 
      ($ (Xb) + (0,0.9,0) $) -- ($ (Xa) + (0,-1.9,0.833) $);

    \draw[ultra thick,color=blue] (1.3,1.5,0) -- (1.3,2,0);
    \draw[ultra thick,color=blue,decorate, decoration={
      brace,amplitude=5pt,raise=2.5pt,mirror
    }] (1.3,1.5,0) -- (1.3,2,0);
    \draw[dashed, thick,color=blue] (1.3,1.5,0) -- ($ (Xb) + (0,1.5,0) $);
    \draw[dashed, thick,color=blue] (1.3,2,0) -- ($ (Xb) + (0,2,0) $);
    \draw[dashed, thick,color=blue] 
      ($ (Xb) + (0,1.5,0) $) -- ($ (Xa) + (0,-1.3,0.833) $);
    \draw[dashed, thick,color=blue] 
      ($ (Xb) + (0,2,0) $) -- ($ (Xa) + (0,-0.8,0.833) $);

    \filldraw[fill=white] (P) circle (1pt);
    \filldraw[fill=black] (X) circle (0.75pt);

    \node[anchor=south east] at (X) {$\vec x$};
    \node[anchor=north west] at (Xa) {$x_1$};

    \node at (0,1,-1.2){$A = \{1\}, \vec\calE = \langle \color{red}\calE_0\color{black},\color{blue}{\calE_2} \color{black} \rangle \in [\aleph_0]^{<\aleph_0} \times [\aleph_2]^{<\aleph_0}$};
    \node at (0,1,-1.5){\colorbox{yellow}{$N_{A,\vec\calE}(\vec x)$}$ = \{\vec y \in X(\aleph_0,\aleph_1,\aleph_2)^- : y_0 \not \in \calE_0, y_1 = x_1, y_2 \not \in \calE_2\}$};
    \end{tikzpicture}

  \caption{The facets $F_A$ and basic open set $\calN(\vec x, \vec \calE)$. Here $\vec\kappa = \langle \aleph_0, \aleph_1, \aleph_2 \rangle$ and $A = \{1\}$.}
  \label{fig:cuboid}
\end{figure}

\newpage
Now, we define the main object of this section.
\begin{definition}\label{def:fub}
  Let $\vec\kappa = \langle \kappa_0,\ldots,\kappa_n \rangle$ be a tuple of infinite cardinals, and fix $-1 \leq k \leq n$. We call an $(n+1)$-element partition $\langle Y_i : i \leq n \rangle$ of $\xkappa^-$ $k$-\textit{sheer} if, for any $A \in [n+1]^{\geq k+1}$, the following criteria hold:
  \begin{enumerate}
    \item Whenever $i \not \in A$ the set $Y_i$ is clopen in $\calD_A$; and
    \item for any $\vec x \in F_A$, there is $\vec \calE \in \prod\limits_{j \not \in A} [\kappa_j]^{<\aleph_0}$ such that
    $$
      \calN(\vec x, \vec \calE) \subset \bigcup_{i \in A} Y_i.
    $$
  \end{enumerate}
\end{definition}
\
\begin{remark}\label{rem:fub}
  We make a few observations.
  \begin{itemize}
    \item Condition (1) will allow us to define continuous functions $f_A$ on $\calD_A$ by defining them piecewise on the $Y_i \cap \calD_A$; when checking continuity, we can separately check each of these pieces when $i \not \in A$ and $\bigcup_{i \in A} Y_i \cap \calD_A$, which is clopen as the complement of their union.
    \item 
      Condition (2) equivalently asserts that for any $A \in [n+1]^{\geq k+1}$, the set
      $\bigcup_{i \in A} Y_i$ contains an open neighborhood of $F_A$, i.e.
      $$
        F_A \subset \text{int}\big(\bigcup_{i \in A} Y_i\big).
      $$
      This reformulation also makes it clearer that $F_A$ is disjoint from $Y_i$ whenever $i\not \in A$. However, it is the more verbose condition (2) which will be most useful to us.
    \item It is immediate from the definition that a $k$-sheer partition is $\ell$-sheer whenever $k \leq \ell$; moreover, if $\vec\kappa \neq \emptyset$ then $(-1)$-sheer partitions never exist, and any partition is $n$-sheer. Consequently, for every nonempty tuple there is a $k_0$ such that a $k$-sheer partition exists if and only if $k_0 \leq k$.
  \end{itemize}
\end{remark}

The special case $k = n-1$ is closely related to the ``almost constant'' functions considered by Aoki:
\begin{proposition}
  Let $\vec \kappa = \langle \kappa_0,\ldots,\kappa_n \rangle$. A partition of $\prod_{i \leq n} \kappa_i$ extends to an $(n-1)$-sheer partition of $\xkappa^-$ if and only if each $Y_i$ is ``finite in $i$'', in the sense that for any $\vec x \in \prod_{j \neq i} \kappa_j$, there are only finitely many $\vec y \in Y_i$ extending $\vec x$.
\end{proposition}

\begin{proof}
  For the forward direction, suppose the $Y_i$ form an $(n-1)$-sheer partition, fix $i \leq n$, and let $\vec x \in \prod_{j \neq i} \kappa_j$. Applying condition (2) to the tuple $\vec x' \in \calD_{n+1 \setminus \{i\}}$ obtained by setting $x_i = \infty$, there is some finite set $\calE_i \subset \kappa_i$ such that $\calN(\vec x', \langle\calE_i\rangle) \subset \bigcup_{j \neq i} Y_j$; this union is equal to the complement of $Y_i$ since the $Y_j$ partition the product. This says exactly that if $\vec z$ extends $\vec x$ and $z_i \not \in \calE_i$, then $\vec z \not \in Y_i$; so if $\vec z \in Y_i$ extends $\vec x$, then $z_i$ lies in the finite set $\calE_i$.

  For the backward direction, suppose the $Y_j$ are each finite in $j$ and partition the product; to extend the $Y_j$ to a partition of $\xkappa^-$, we can put each $\vec x$ with $x_j = \infty$ in any $Y_i$ with $x_i \neq \infty$. Again let $i \leq n$ and let $x \in F_{n+1 \setminus \{i\}}$. By hypothesis, there are finitely many $\vec y \in Y_i$ which agree with $\vec x$ on the coordinates other than $i$; we let $\calE_i$ be the set of their $i$-coordinates. Then $\calN(\vec x, \langle \calE_i \rangle) \subset Y_i^c = \bigcup_{j \neq i} Y_j$; this proves condition (2). Condition (1) also follows, since $Y_i \cap \calD_{n+1 \setminus \{i\}}$ consists of isolated points and the above construction shows its complement is open.
\end{proof}

We defer discussion of the (non-)existence of sheer partitions to the end of the section, focusing instead on their uses. 

\
We will actually use the following strengthening of condition (1), which uses condition (2):

\begin{lemma}\label{lem:upgrade}
  Let $\langle Y_i : i \leq n\rangle$ be a $k$-sheer partition of $\xkappa^-$.
  Then for any $A \in [n+1]^{\geq k+1}$ and any $i \not \in A$,
  $$
    \calD_A \cap Y_i = \calD_{A \cup \{i\}} \cap Y_i.
  $$
\end{lemma}

\begin{proof}
  Fix $A \in [n+1]^{\geq k+1}$ and $i \not \in A$. For any $B \supseteq A$ not containing $i$ we have by Remark \ref{rem:fub} that $F_B$ and $Y_i$ are disjoint; that is,
  $$
    \coprod_{\substack{B \supset A \\ i \not \in B}} F_B \cap Y_i = \emptyset.
  $$
  Adding the $F_B$ for $B \supset A, i \in B$ to both sides, we get
  $$
    \coprod_{B \supset A} F_B \cap Y_i = \coprod_{\substack{B \supset A \\ i \in B}} F_B \cap Y_{i}.
  $$
  which becomes $\calD_A \cap Y_i = \calD_{A \cup \{i\}} \cap Y_i$ by the remark in Definition \ref{def:basis}.
\end{proof}

The purpose of Lemma \ref{lem:upgrade} is to ``upgrade'' a continuous function 
$\calD_{A \cup \{i\}}(\vec\kappa) \to \bbZ_2$ to a continuous function 
$\calD_A(\vec\kappa) \to \bbZ_2$, when restricted to the subspace
$\calD_{A \cup \{i\}} \cap Y_i$. This, in tandem with condition (1), enables piecewise definitions of cocycles using
functions which are not continuous on the whole of $\calD_A$.

As a first example of this, we construct a trivializer of a $(k+1)$-cocycle using a $k$-sheer partition:

\begin{proposition}\label{prop:trivialize}
  Let $\langle Y_i : i \leq n\rangle$ be a $k$-sheer partition of $\xkappa^-$, and suppose $\vec{f} \in \prod_{A \in [n+1]^{k+2}} \calC_A(\vec\kappa)$ is a tuple of continuous functions.
  For any $B \in [n+1]^{k+1}$, any $i \leq n$, and any $\vec x \in Y_i$, define
  $$
    g_B(\vec x) = 
    \begin{cases}
      f_{B \cup \{i\}}(\vec x) &\t{if } i \not \in B, \\
      0 &\t{otherwise.}
    \end{cases}
  $$
  Then $g_B \in \calC_B(\vec\kappa)$. Moreover, if $\vec f$ is a cocycle, then
  $d\vec g = \vec f$.
\end{proposition}

\begin{proof}
  As remarked above, we can separately check the continuity of $g_B$ restricted to each $\calD_B \cap Y_i$ when $i \not \in B$, and on the remaining part $\bigcup_{i \in B} Y_i$. In the first case, continuity follows from the continuity of $f_{B \cup \{i\}}$ on $\calD_{B \cup \{i\}} \cap Y_i$ and Lemma \ref{lem:upgrade}. On the remaining part $\bigcup_{i \in B} Y_i \cap \calD_B$, $g_B$ is identically zero and thus certainly continuous.

  To see that $d\vec g = \vec f$, we first note that for any $A \in [n+1]^{k+2}$ we have
  $$
    (d\vec g)_{A} = \sum_{i \in A} g_{A \setminus \{i\}}.
  $$
  Fix $j \leq n$ and $\vec x \in Y_j$. If $j \in A$ then 
  $g_{A \setminus \{i\}}(\vec x) = 0$ unless $i = j$, in which case it is 
  $f_A(\vec x)$. If $j \not \in A$ then the sum is 
  $\sum_{i \in A} f_{A \setminus \{i\} \cup \{j\}}(\vec x)$ which is equal to 
  $f_A(\vec x)$ by the cocycle property of $\vec f$.
  Thus $d\vec g = \vec f$ as desired.
\end{proof}

\newpage
Along with the third point in Remark \ref{rem:fub}, this gives the following corollary:
\begin{corollary}\label{cor:fub-vanish}
  If there is a $k$-sheer partition of $\xkappa^-$ then $H^{\ell}(\xkappa^-,\underline{\bbZ_2}) = 0$ for all $\ell > k$.
\end{corollary}

Figure \ref{fig:chart1} shows a table view of the piecewise definition of the $g_B$ for 
the case $k=1, n=3$.

\begin{figure}[htbp]
  \centering
  \begin{tabular}{ c|c|c|c|c } 
             & $Y_0$     & $Y_1$     & $Y_2$     & $Y_3$     \\ \hline
    $g_{01}$ & $0$       & $0$       & $f_{012}$ & $f_{013}$ \\ 
    $g_{02}$ & $0$       & $f_{012}$ & $0$       & $f_{023}$ \\ 
    $g_{03}$ & $0$       & $f_{013}$ & $f_{023}$ & $0$       \\
    $g_{12}$ & $f_{012}$ & $0$       & $0$       & $f_{123}$ \\
    $g_{13}$ & $f_{013}$ & $0$       & $f_{123}$ & $0$       \\
    $g_{23}$ & $f_{023}$ & $f_{123}$ & $0$       & $0$       \\
   \end{tabular}
  \caption{Given a 2-cocycle $\vec f = \langle f_{012}, f_{013}, f_{023}, f_{123} \rangle$ and a $k$-sheer partition $\langle Y_0,Y_1,Y_2,Y_3 \rangle$, the above piecewise definition of $\langle g_B : B \in [4]^2\rangle$ defines a trivializer of $\vec f$. (For brevity the columns are labeled $Y_i$ rather than $\calD_A \cap Y_i$.)}
  \label{fig:chart1}
\end{figure}

We next use a $k$-sheer partition of $\xkappa^-$ to show surjectivity of the restriction map from $H^k(\xkappa^-,\underline{\bbZ_2})$ to $H^k(\xlambda^-,\underline{\bbZ_2})$ when $\vec{\lambda} \leq \vec\kappa$ coordinate-wise. A priori, one might try to extend a cocycle by treating large enough coordinates as infinity and mapping unrestricted values to zero, as follows:
\begin{definition}\label{def:naive-ext}
  For $\vec{\lambda} \leq \vec\kappa$ coordinate-wise and a function $f_A \in \calC_A(\vec\lambda)$, define the \textit{naive extension} of $f_A$ by
  $$
    f_A^\infty(\vec x) = 
    \begin{cases}
      f_A(\vec x), &\t{ if all } x_i \in \lambda_i, \\
      f_A(\vec x^\infty), &\t{ if } x_i \not \in \lambda_i \t{ only for } i \not \in A, \\
      0,  &\t{ otherwise,}
    \end{cases}
  $$
  where $\vec x^\infty$ is obtained from $\vec x$ by setting all $x_i \geq \lambda_i$ equal to $\infty$.

  This extends to a map on cochains $\vec f \mapsto \vec f^\infty$ by acting on the individual functions.
\end{definition}

Figure \ref{fig:naive} shows an example of the naive extension of a cochain from $\xlambda^-$ to $\xkappa^-$ for $\vec\lambda \leq \vec\kappa$.

\begin{figure}[htbp]
  \centering
  \tdplotsetmaincoords{70}{110}
  \begin{tikzpicture}[scale=2,tdplot_main_coords]
    \coordinate (O) at (0,0,0);
    \tdplotsetcoord{P}{1.5}{55}{45}

    \fill [yellow,opacity=0.5](Pxy) -- (Py) -- (Pyz)  -- (P);

    \draw[thick,->] (O) -- (0,0,1) node[anchor=south,xshift=7pt]{$0 \scriptstyle{(\aleph_0)}$};
    \draw[thick,->] (O) -- (1.5,0,0) node[anchor=north east]{$1 \scriptstyle{(\aleph_0)}$};
    \draw[thick,->] (O) -- (0,1,0) node[anchor=north west]{$2 \scriptstyle{(\aleph_0)}$};
    \draw[thick] (O) -- (Px);
    \draw[thick] (O) -- (Py);
    \draw[thick] (O) -- (Pz);
    \draw[thick] (Px) -- (Pxy) -- (Py);
    \draw[thick] (Px) -- (Pxz) -- (Pz);
    \draw[thick] (Py) -- (Pyz) -- (Pz);
    \draw[thick] (Pxy) -- (P);
    \draw[thick] (Pyz) -- (P);
    \draw[thick] (Pxz) -- (P);

    \filldraw[fill=white] (P) circle (1pt);

  \end{tikzpicture}
  \begin{tikzpicture}[scale=2,tdplot_main_coords]
    \coordinate (O) at (0,0,0);
    \tdplotsetcoord{P}{2}{65}{65}

    \fill [yellow,opacity=0.5](Pxy) -- (Py) -- (Pyz)  -- (P);

    \draw[color=blue,thick,->] (0.2,0.75,0.75) -- (0.2,1.65,0.75);
    \draw[color=blue,thick,->] (0.4,0.75,0.75) -- (0.4,1.65,0.75);
    \draw[color=blue,thick,->] (0.6,0.75,0.75) -- (0.6,1.65,0.75);
    \draw[color=blue,thick,->] (0.8,0.75,0.75) -- (0.8,1.65,0.75);
    \draw[color=blue,thick,->] (0.2,0.75,0.45) -- (0.2,1.65,0.45);
    \draw[color=blue,thick,->] (0.4,0.75,0.45) -- (0.4,1.65,0.45);
    \draw[color=blue,thick,->] (0.6,0.75,0.45) -- (0.6,1.65,0.45);
    \draw[color=blue,thick,->] (0.8,0.75,0.45) -- (0.8,1.65,0.45);
    \draw[color=blue,thick,->] (0.2,0.75,0.15) -- (0.2,1.65,0.15);
    \draw[color=blue,thick,->] (0.4,0.75,0.15) -- (0.4,1.65,0.15);
    \draw[color=blue,thick,->] (0.6,0.75,0.15) -- (0.6,1.65,0.15);
    \draw[color=blue,thick,->] (0.8,0.75,0.15) -- (0.8,1.65,0.15);

    \draw[thick,->] (O) -- (0,0,1) node[anchor=south,xshift=7pt]{$0 \scriptstyle{(\aleph_0)}$};
    \draw[thick,->] (O) -- (1.5,0,0) node[anchor=north east]{$1 \scriptstyle{(\aleph_0)}$};
    \draw[thick,->] (O) -- (0,2,0) node[anchor=north west]{$2 \scriptstyle{(\aleph_1)}$};
    \draw[thick] (O) -- (Px);
    \draw[thick] (O) -- (Py);
    \draw[thick] (O) -- (Pz);
    \draw[thick] (Px) -- (Pxy) -- (Py);
    \draw[thick] (Px) -- (Pxz) -- (Pz);
    \draw[thick] (Py) -- (Pyz) -- (Pz);
    \draw[thick] (Pxy) -- (P);
    \draw[thick] (Pyz) -- (P);
    \draw[thick] (Pxz) -- (P);

    \filldraw[fill=white] (P) circle (1pt);

    \draw[thick] (0,.75,-0.1) -- (0,.75,0.1) node[anchor=north east,scale=0.5,yshift=5pt] {$\aleph_0$};
    \draw[dashed] (0,.75,0) -- (0,.75,.85) -- (0.8,0.75,0.85) -- (0.8,0.75,0) -- (0,0.75,0) ;
  \end{tikzpicture}
  \begin{tikzpicture}[scale=2,tdplot_main_coords]
    \coordinate (O) at (0,0,0);
    \tdplotsetcoord{P}{1.5}{55}{45}
  
    \fill [yellow,opacity=0.5](Pxy) -- (Px) -- (Pxz)  -- (P);
  
    \draw[thick,->] (O) -- (0,0,1) node[anchor=south,xshift=7pt]{$0 \scriptstyle{(\aleph_0)}$};
    \draw[thick,->] (O) -- (1.5,0,0) node[anchor=north east]{$1 \scriptstyle{(\aleph_0)}$};
    \draw[thick,->] (O) -- (0,1,0) node[anchor=north west]{$2 \scriptstyle{(\aleph_0)}$};
    \draw[thick] (O) -- (Px);
    \draw[thick] (O) -- (Py);
    \draw[thick] (O) -- (Pz);
    \draw[thick] (Px) -- (Pxy) -- (Py);
    \draw[thick] (Px) -- (Pxz) -- (Pz);
    \draw[thick] (Py) -- (Pyz) -- (Pz);
    \draw[thick] (Pxy) -- (P);
    \draw[thick] (Pyz) -- (P);
    \draw[thick] (Pxz) -- (P);
  
    \filldraw[fill=white] (P) circle (1pt);
  
    \end{tikzpicture}
  \begin{tikzpicture}[scale=2,tdplot_main_coords]
    \coordinate (O) at (0,0,0);
    \tdplotsetcoord{P}{2}{65}{65}
  
    \fill [yellow,opacity=0.5](Px) -- ($(Px) + (0,0.75,0)$) -- ($(Pxz) + (0,0.75,0)$)  -- (Pxz);
    \fill [cyan,opacity=0.5]($(Px) + (0,0.75,0)$) -- (Pxy) -- (P)  -- ($(Pxz) + (0,0.75,0)$);

    \draw[thick,->] (O) -- (0,0,1) node[anchor=south,xshift=7pt]{$0 \scriptstyle{(\aleph_0)}$};
    \draw[thick,->] (O) -- (1.5,0,0) node[anchor=north east]{$1 \scriptstyle{(\aleph_0)}$};
    \draw[thick,->] (O) -- (0,2,0) node[anchor=north west]{$2 \scriptstyle{(\aleph_1)}$};
    \draw[thick] (O) -- (Px);
    \draw[thick] (O) -- (Py);
    \draw[thick] (O) -- (Pz);
    \draw[thick] (Px) -- (Pxy) -- (Py);
    \draw[thick] (Px) -- (Pxz) -- (Pz);
    \draw[thick] (Py) -- (Pyz) -- (Pz);
    \draw[thick] (Pxy) -- (P);
    \draw[thick] (Pyz) -- (P);
    \draw[thick] (Pxz) -- (P);
  
    \filldraw[fill=white] (P) circle (1pt);

    \draw[thick] (0,.75,-0.1) -- (0,.75,0.1) node[anchor=north east,scale=0.5,yshift=5pt] {$\aleph_0$};
    \draw[dashed] (0,.75,0) -- (0,.75,.85) -- (0.8,0.75,0.85) -- (0.8,0.75,0) -- (0,0.75,0) ;

    \node at (0,1,0.3) {\large{0}};
    \end{tikzpicture}
  \begin{tikzpicture}[scale=2,tdplot_main_coords]
    \coordinate (O) at (0,0,0);
    \tdplotsetcoord{P}{1.5}{55}{45}

    \fill [yellow,opacity=0.5](Pyz) -- (Pz) -- (Pxz)  -- (P);

    \draw[thick,->] (O) -- (0,0,1) node[anchor=south,xshift=7pt]{$0 \scriptstyle{(\aleph_0)}$};
    \draw[thick,->] (O) -- (1.5,0,0) node[anchor=north east]{$1 \scriptstyle{(\aleph_0)}$};
    \draw[thick,->] (O) -- (0,1,0) node[anchor=north west]{$2 \scriptstyle{(\aleph_0)}$};
    \draw[thick] (O) -- (Px);
    \draw[thick] (O) -- (Py);
    \draw[thick] (O) -- (Pz);
    \draw[thick] (Px) -- (Pxy) -- (Py);
    \draw[thick] (Px) -- (Pxz) -- (Pz);
    \draw[thick] (Py) -- (Pyz) -- (Pz);
    \draw[thick] (Pxy) -- (P);
    \draw[thick] (Pyz) -- (P);
    \draw[thick] (Pxz) -- (P);

    \filldraw[fill=white] (P) circle (1pt);

    \end{tikzpicture}
  \begin{tikzpicture}[scale=2,tdplot_main_coords]
    \coordinate (O) at (0,0,0);
    \tdplotsetcoord{P}{2}{65}{65}
  
    \fill [yellow,opacity=0.5](Pz) -- ($(Pxz)$) -- ($(Pxz) + (0,0.75,0)$)  -- ($(Pz) + (0,0.75,0)$);
    \fill [cyan,opacity=0.5]($(Pz) + (0,0.75,0)$) -- (Pyz) -- (P)  -- ($(Pxz) + (0,0.75,0)$);
  
    \draw[thick,->] (O) -- (0,0,1) node[anchor=south,xshift=7pt]{$0 \scriptstyle{(\aleph_0)}$};
    \draw[thick,->] (O) -- (1.5,0,0) node[anchor=north east]{$1 \scriptstyle{(\aleph_0)}$};
    \draw[thick,->] (O) -- (0,2,0) node[anchor=north west]{$2 \scriptstyle{(\aleph_1)}$};
    \draw[thick] (O) -- (Px);
    \draw[thick] (O) -- (Py);
    \draw[thick] (O) -- (Pz);
    \draw[thick] (Px) -- (Pxy) -- (Py);
    \draw[thick] (Px) -- (Pxz) -- (Pz);
    \draw[thick] (Py) -- (Pyz) -- (Pz);
    \draw[thick] (Pxy) -- (P);
    \draw[thick] (Pyz) -- (P);
    \draw[thick] (Pxz) -- (P);
  
    \filldraw[fill=white] (P) circle (1pt);

    \draw[thick] (0,.75,-0.1) -- (0,.75,0.1) node[anchor=north east,scale=0.5,yshift=5pt] {$\aleph_0$};
    \draw[dashed] (0,.75,0) -- (0,.75,.85) -- (0.8,0.75,0.85) -- (0.8,0.75,0) -- (0,0.75,0) ;

    \node at (0,1,0.3) {\large{0}};
  \end{tikzpicture}

  \caption{The naive extension of a 1-cochain $\langle f_{01},f_{02},f_{12} \rangle$ defined on $X(\aleph_0,\aleph_0,\aleph_0)^-$ (left) to $\langle f_{01}^\infty,f_{02}^\infty,f_{12}^\infty \rangle$ defined on $X(\aleph_0,\aleph_0,\aleph_1)^-$ (right). Notice that the latter is not necessarily a cocycle.}
  \label{fig:naive}
\end{figure}

This definition indeed maps cochains to cochains, and in the case $k = n$ it maps cocycles to cocycles. Since a restriction of a trivializer for $\vec f^\infty$ to $X(\lambda)^-$ yields a trivializer for $\vec f$, it follows that for top cohomology
$$
  H^n(\xlambda^-,\underline{\bbZ_2}) \inj H^n(\xkappa^-,\underline{\bbZ_2})
  \quad \t{whenever} \quad
  \vec \lambda \leq \vec \kappa.
$$
\newpage
However, the naive extension fails in the case of intermediate cohomology:

\begin{example}
  Let $\vec\lambda = \langle \aleph_0,\aleph_0,\aleph_0 \rangle, \vec\kappa = \langle \aleph_0,\aleph_0,\aleph_1\rangle,$ and let $\vec f \in H^1(\xlambda^-,\underline{\bbZ_2})$ be defined by 
  \begin{align*}
    f_{01}(\vec x) &= f_{02}(\vec x) = 1, \\
    f_{12}(\vec x) &= 0. \\
  \end{align*}
  The naive extension maps $f_{01}$ and $f_{12}$ to the constant-1 function and constant-0 function, respectively, on $\xkappa$, but $f_{02}^\infty(\vec x) = 0$ when $x_2 \geq \aleph_0$. Thus the sum $f_{01}^\infty + f_{02}^\infty + f_{12}^\infty$ is not zero, so $\vec f^\infty$ is not a cocycle.
\end{example}

\newpage
Thus we will need to modify the cochain $\vec f^\infty$ to a cocycle. In fact we will show that, given the existence of a $k$-sheer partition, any $k$-cochain can be modified to a $k$-cocycle without affecting its values at points where it satisfies the cocycle condition.

\begin{definition}\label{def:modification}
  Let $\langle Y_i : i \leq n\rangle$ be a $k$-sheer partition of $\xkappa^-$, and let $\vec{f} \in C^k(\xkappa^-,\underline{\bbZ_2})$ be a $k$-cochain. Define the \textit{modification} of $\vec{f}$ as follows: for any $A \in [n+1]^{k+1}$ and any $\vec x \in \calD_A \cap Y_i$, define
  $$
    f^{\Mod}_A(\vec x) = 
    \begin{cases}
      f_A(\vec x) &\t{if } i \in A, \\
      \sum_{j \in A} f_{(A \cup \{i\}) \setminus \{j\}}(\vec x) &\t{if } i \not \in A.
    \end{cases}
  $$
\end{definition}

\begin{lemma}\label{lem:mod-cts}
  For any cochain $\vec f$, the modified tuple $\vec f^{\Mod}$ is a cocycle. Moreover, for any $A \in [n+1]^{k+1}$, any $i \not \in A$, and any $\vec x \in \calD_A \cap Y_i$, we have
  \begin{equation}\label{eqn:mod}
    f_A^{\Mod}(\vec x) = f_A(\vec x) + (d\vec f)_{A \cup \{i\}}(\vec x).
  \end{equation}
  In particular $f_A^{\Mod}$ agrees with $f_A$ wherever the latter satisfies the cocycle condition.
\end{lemma}
Note that the right side of (\ref{eqn:mod}) is well-defined: by Lemma \ref{lem:upgrade}, $\calD_A \cap Y_i = \calD_{A \cup \{i\}} \cap Y_i$.

\begin{proof}
  We first show that each $f^{\Mod}_A$ is continuous; by Remark \ref{rem:fub} it suffices to check on each $\calD_A \cap Y_i$ when $i \not \in A$, and on the remaining part $\bigcup_{i \in A} Y_i.$ On the latter subdomain, $f_A^{\Mod}$ is identical to $f_A$, so continuity is immediate; so fix $\vec x \in \calD_A \cap Y_i$ where $i \not \in A$. By Lemma \ref{lem:upgrade} we have for each $j \in A$
    $$
      \calD_{A \setminus \{j\}} \cap Y_i = \calD_{(A \cup \{i\}) \setminus \{j\}} \cap Y_i.
    $$
    \noindent Since $f_{(A \cup \{i\}) \setminus \{j\}}$ is continuous on $\calD_{(A \cup \{i\}) \setminus \{j\}} \cap Y_i$, there is a neighborhood $U_j \subset \calD_{A \setminus \{j\}} \cap Y_i$ such that $f_{(A \cup \{i\}) \setminus \{j\}}$ is constant on $U_j$. Let $U = \calD_A \cap \bigcap_{j \in A} U_j$; then in $U$ we have that each $f_{(A \cup \{i\}) \setminus \{j\}}$ is constant, so that
    $$
      f^{\Mod}_A= \sum_{j \in A} f_{(A \cup \{i\}) \setminus \{j\}}
    $$
    is constant on $U$.
  
  So $\vec f^{\Mod}$ is a cochain. To check the cocycle condition, assume $k \leq n-1$ (if $k = n$ then the cocycle condition is automatic) and choose $B \in [n+1]^{k+2}$. Again take $\vec x \in Y_i$ and case split on whether $i \in B$:
  \begin{itemize}
    \item If $i \in B$, then
    \begin{align*}
      \sum_{j \in B} f^{\Mod}_{B \setminus \{j\}}(\vec x) &= f^{\Mod}_{B \setminus \{i\}}(\vec x) + \sum_{j \in B \setminus \{i\}} f^{\Mod}_{B \setminus \{j\}}(\vec x) \\
      &= \sum_{j \in B \setminus \{i\}} f_{B \setminus \{j\}}(\vec x) + \sum_{j \in B \setminus \{i\}} f_{B \setminus \{j\}}(\vec x) \\
      &= 0;
    \end{align*}
    \item If $i \not \in B$ then
    \begin{align*}
      \sum_{j \in B} f^{\Mod}_{B \setminus \{j\}}(\vec x) &= \sum_{j \in B}\sum_{\ell \in B \setminus \{j\}}f_{B \setminus \{j,\ell\} \cup \{i\}}(\vec x) \\
      &= 0.
    \end{align*}
  \end{itemize}
  Lastly, to see the identity (\ref{eqn:mod}), note that by definition
  \begin{align*}
    (d\vec f)_{A \cup \{i\}}(\vec x) &= \sum_{j \in A \cup \{i\}} f_{(A \cup \{i\}) \setminus \{j\}}(\vec x) \\
    &= f_A(\vec x) + \sum_{j \in A} f_{(A \cup \{i\}) \setminus \{j\}}(\vec x) \\
    &= f_A(\vec x) + f_A^{\Mod}(\vec x),
  \end{align*}
  from which (\ref{eqn:mod}) follows after adding $f_A(\vec x)$ to both sides.
\end{proof}

Figure \ref{fig:chart2} shows a table view of the piecewise definition of $\vec f^{\Mod}$ for 
the case $k=1, n=3$.

\begin{figure}[htbp]
  \centering
  \begin{tabular}{ c|c|c|c|c } 
                    & $Y_0$           & $Y_1$           & $Y_2$             & $Y_3$           \\ \hline
    $f^{\Mod}_{01}$ & $f_{01}$        & $f_{01}$        & $f_{02}+f_{12}$   & $f_{03}+f_{13}$ \\ 
    $f^{\Mod}_{02}$ & $f_{02}$        & $f_{01}+f_{12}$ & $f_{02}$          & $f_{03}+f_{23}$ \\ 
    $f^{\Mod}_{03}$ & $f_{03}$        & $f_{01}+f_{13}$ & $f_{02}+f_{23}$   & $f_{03}$        \\
    $f^{\Mod}_{12}$ & $f_{01}+f_{02}$ & $f_{12}$        & $f_{12}$          & $f_{13}+f_{23}$ \\
    $f^{\Mod}_{13}$ & $f_{01}+f_{03}$ & $f_{13}$        & $f_{12}+f_{23}$   & $f_{13}$        \\
    $f^{\Mod}_{23}$ & $f_{02}+f_{03}$ & $f_{12}+f_{13}$  & $f_{23}$         & $f_{23}$        \\
   \end{tabular}
  \caption{Given a 1-cochain $\vec f = \langle f_{01}, f_{02}, f_{03}, f_{12},f_{13},f_{23} \rangle$ and a 1-sheer partition $\langle Y_0,Y_1,Y_2,Y_3 \rangle$, the above piecewise definition of $\langle f^{\Mod}_A : A \in [4]^2\rangle$ modifies $\vec f$ to a cocycle without affecting its values at points where $\vec f$ itself satisfies the cocycle condition. (Again, for brevity the columns are labeled $Y_i$ rather than $\calD_A \cap Y_i$.)}
  \label{fig:chart2}
\end{figure}

\begin{proposition}\label{prop:embedding}
  Suppose $\xkappa^-$ admits a $k$-sheer partition.
  Then for any $\vec{\lambda} \leq \vec\kappa$ coordinate-wise, the restriction
  $$
    H^k(\xkappa^-,\underline{\bbZ_2}) \to H^k(\xlambda^-,\underline{\bbZ_2})
  $$
  is surjective.
\end{proposition}
\begin{proof}
  Given a $k$-cocycle $\vec f$ defined on $\xlambda^-$, perform the naive extension to obtain a cochain on $\xkappa^-$, then take its modification to obtain a cocycle. By Lemma \ref{lem:mod-cts}, this restricts to $\vec f$ on $\xlambda^-$, since $\vec f^\infty$ restricts to $\vec f$ and the latter is a cocycle.
\end{proof}

It now remains to determine when sheer partitions exist. As we saw above, the interesting cases are $0 \leq k < n$. We first consider when $k < n-1$.

\begin{proposition}\label{prop:fub-iff-zero}
  Let $\vec\kappa$ be a weakly increasing tuple of infinite cardinals, $0 \leq k < n - 1$. Then a $k$-sheer partition of $\xkappa^-$ exists if and only if $\kappa_{n-k} = \aleph_0$, that is, the first $n-k+1$ entries of $\vec\kappa$ are all $\aleph_0$.
\end{proposition}

\begin{proof}
  For the forward direction, suppose $\vec\kappa$ admits a $k$-sheer partition. Then we simultaneously have that (i) $\vec\kappa$ also admits a $(k+1)$-sheer partition (Remark \ref{rem:fub}), and that (ii) $H^{k+1}(\xkappa^-,\underline{\bbZ_2}) = 0$ (Corollary \ref{cor:fub-vanish}). Since restriction to $\xlambda^-$ is surjective, no $\vec\lambda \leq \vec\kappa$ has nonzero $(k+1)$-cohomology. But now consider the tuple
  $$
  \vec\lambda = \langle\underbrace{\aleph_0,\aleph_0,\ldots,\aleph_0}_{n-k}, \underbrace{\aleph_1,\aleph_1,\ldots,\aleph_1}_{k+1}\rangle.
  $$
  By Theorem \ref{thm:main1}, $\vec\lambda$ has nonzero $(k+1)$-cohomology (note $k+1 < n$ by the assumption $k < n-1$). So it cannot lie below $\vec\kappa$ coordinatewise, which is equivalent to $\kappa_{n-k} = \aleph_0$.

  For the backward direction, suppose $\kappa_{n-k} = \aleph_0$, and for $i \leq n$ set
  $$
    Y_i = \{\vec x \in \xkappa : \min \vec x = x_i\},
  $$
  where $\infty$ is treated as larger than any ordinal.\footnote{Ties can be broken by an arbitrary scheme; for concreteness, we can prefer earlier coordinates, but we will never use this.} These sets are disjoint, and since any tuple containing a non-$\infty$ coordinate has a minimum, they cover $\xkappa^-$. It remains to check conditions (1),(2); for this fix $A \in [n+1]^{\geq k+1}$. Note that there are only $k$ coordinates in the set $\{n-k+1, n-k+2, \ldots, n-1, n\}$, and $|A| > k$; thus there must be $i_0 \in A$ such that $i_0 \leq n-k$, and so by hypothesis $\kappa_{i_0} = \aleph_0$. So if $\vec x \in \calD_A$, $\min \vec x < \aleph_0$. This lets us quickly verify (1): namely, for $j$ such that $x_j = \infty$, set
  $$
  \calE_j = \{0,1,\ldots,\min\vec x\} \in [\kappa_j]^{<\aleph_0}.
  $$

  Then $\calN(\vec x, \vec\calE)$ is a neighborhood of $\vec x$, all of whose points attain their minimum at the same coordinate. If $\vec x \in Y_i$ then this neighborhood lies within $Y_i$ and if $\vec x \not \in Y_i$ it lies outside $Y_i$. So $Y_i$ is clopen in $\calD_A$.
  A similar approach shows (2): fix $\vec x \in F_A$ and for $j \not \in A$ set $\calE_j = \{0,1,\ldots,x_{i_0}\}$. If $\vec y \in \calN(\vec x,\vec\calE)$, then for every $j \not \in A$ we have $y_j > x_{i_0} = y_{i_0}$; so $\min \vec y \neq y_j$ and thus $\vec y \not \in Y_j$. We conclude that the minimum of $\vec y$ is among the $y_i$ for $i \in A$, i.e.,
  $$
    \calN(\vec x,\vec\calE) \subset \bigcup_{i \in A} Y_i,
  $$
  completing the proof of (2).
\end{proof}

The remaining case is $k=n-1$. Here the backward direction of Proposition \ref{prop:fub-iff-zero} still holds: if $\kappa_1 = \aleph_0$ then the above defines an $(n-1)$-sheer partition, thus trivializing the top cohomology of $\xkappa^-$. On the other hand, Aoki showed that if $\kappa_i \geq \aleph_i$ for all $i$, then $H^n(\xkappa^-,\underline{\bbZ_2}) \neq 0$. Our converse to Aoki's result comes from the requisite strengthening of the above forward direction for $k = n-1$:

\begin{proposition}\label{prop:aoki-converse-fub}
  Suppose there is $i_0 \leq n$ such that $\kappa_{i_0} < \aleph_{i_0}$. Then an $(n-1)$-sheer partition exists for $\vec\kappa$.
\end{proposition}

We will need a few ingredients for the proof.

\begin{itemize}
  \item First: we fix once and for all a family of injections 
  $$
    \langle e_\beta : \beta \inj |\beta| \bigm| \beta < \kappa_n \rangle.
  $$
  
  For a subset $A \subset n+1$ of size $\geq 2$ and a point $\vec x \in \prod_{i \in A} \kappa_i$, we define the \textit{virtual minimum} $\vmin\vec x$ of $\vec x$ (with respect to the $e_\beta$) inductively on the size of $A$, as follows:
    \begin{itemize}
      \item If $|A|=2$, then $\vmin \vec x = \min \vec x$.
      \item If $|A| \geq 3$, then the virtual minimum of $\vec x$ is defined as
      $$
        \vmin\vec x := \vmin\langle e_{\max\vec x}(x_i) : i \in A, i \neq i_{\max}\rangle,
      $$
      where $i_{\max}$ is the index of $\max{}\vec x$, with ties broken lexicographically.
    \end{itemize}
    Algorithmically: remove the maximum element from the tuple, apply its corresponding injection to the remaining elements, then repeat until the tuple consists of just two elements, at which point the virtual minimum is just the minimum.
  
    Certainly this process terminates, and designates exactly one element of the tuple as its virtual minimum. We denote the index of $\vmin\vec x$ by $\iota_{\vmin}(\vec x)$.
    \item Second: as noted in Remark \ref{rem:fub}, in the case $k = n-1$, conditions (1) and (2) of Definition \ref{def:fub} reduce to the single condition
    \begin{quote}
      $(3)$ For every $j \leq n$ and every $\vec x^j = \langle x_i \in \kappa_i : i \neq j\rangle$, there are only finitely many possible values of $x_j \in \kappa_j$ such that the resulting tuple $\vec x$ lies in $Y_j$.
    \end{quote}
    So we will need only construct a partition $\langle Y_i \rangle$ satisfying (3).
    \item 
    Third: we in fact can assume $i_0 = n$, i.e., $\kappa_n < \aleph_n$. This is because of the general fact that an $(i_0-1)$-sheer partition of $X(\kappa_0,\ldots,\kappa_{i_0})^-$ yields an $(n-1)$-sheer partition of $\xkappa^-$ by setting $Y_i$ for $i \leq i_0$ equal to the preimage of the respective projection maps and the remaining $Y_i$ equal to the empty set.
\end{itemize}

With these in hand, we are ready to prove Proposition \ref{prop:aoki-converse-fub}.

\begin{proof}[Proof of Proposition \ref{prop:aoki-converse-fub}]
  We first claim that $\kappa_n < \aleph_n$ (i.e., $i_0 = n$ in the notation of the statement) implies $\vmin\vec x < \aleph_0$. We prove this by induction on the length $n$:
  \begin{itemize}
    \item If $n = 1$ then $\vec \kappa = \langle \aleph_0,\aleph_0 \rangle$ and the claim is immediate.
    \item Suppose the claim holds for tuples of length $n$, and let $\vec\kappa = \langle \kappa_i: i \leq n+1 \rangle$ with $\kappa_{n+1} < \aleph_{n+1}$. Then, given $\vec x \in \prod_{i \leq n+1} \kappa_i$, we have $|\max \vec x| < \kappa_{n+1}$, so that for each $i \neq i_{\max}$ we have 
    $$
      e_{\max\vec x}(x_i) < |\max \vec x\, | < \kappa_{n+1} < \aleph_{n+1}.
    $$
    Therefore $\langle e_{\max\vec x}(x_i) : i \neq i_{\max} \rangle$ is a tuple of length $n$ whose maximum is below $\aleph_n$; the inductive hypothesis then implies its virtual minimum, and thus the virtual minimum of $\vec x$, is finite.
  \end{itemize}
  Now, for $j \leq n$ define
  $$
    Y_j := \left\{ \vec x \in \prod_{i \leq n} \kappa_i : \iota_{\vmin}(\vec x) = j \right\}.
  $$
  We claim that if the virtual minimum of $\vec x$ is finite, then the $Y_j$ satisfy (3). We again prove this by induction on the length of the tuple:
  \begin{itemize}
    \item If $n = 1$ then again $\vec\kappa = \langle \aleph_0,\aleph_0 \rangle$. Fixing one of these two coordinates,  we see that only finitely many choices of the other coordinate enable it to be the (virtual) minimum of $\vec x$.
    \item Now suppose the claim holds for tuples of length $n$, let $\vec \kappa = \langle\kappa_0,\ldots,\kappa_{n+1}\rangle$ such that every $\vec x \in \prod_i \kappa_i$ has finite virtual minimum, and fix all $x_i \in \kappa_i$ except a single coordinate $x_j$. Suppose we have chosen $x_j$ such that $\iota_{\vmin}(\vec x) = j$. Then $j$ is also the virtual minimal index of $\langle e_{\max\vec x}(x_i) : i \neq i_{\max} \rangle$, a tuple of length $n$ with finite virtual minimum. Thus by induction $e_{\max\vec x}(x_j)$ comes from a finite set of possibilities, which pulls back to a finite set of possibilities for $x_j$ since the $e_\beta$ are injective.
  \end{itemize}
  So the $\langle Y_j \rangle$ form an $(n-1)$-sheer partition.
\end{proof}

\begin{corollary}
  Let $\vec\kappa$ be a tuple of infinite cardinals such that $\kappa_i < \aleph_i$ for some $i \leq n$. Then $H^n(\xkappa^-,\underline{\bbZ_2}) = 0$.
\end{corollary}
\begin{proof}
  By Corollary \ref{cor:fub-vanish} and Proposition \ref{prop:aoki-converse-fub}.
\end{proof}

\section{Cocycles as pointwise limits of coboundaries}\label{sec:main3}

In this section we prove Theorem \ref{thm:main3}. In the process, we will describe a technique for reducing questions about intermediate cohomology $(0 < k < n)$ to questions about top cohomology $(k = n)$. We first note the existence of a certain restriction map between \v{C}ech complexes, which is well-defined on cohomology:

\begin{definition}\label{def:restriction}
  Let $\vec\kappa$ be a tuple of infinite cardinals of length $n+1$, $0 < k < n$. Fix also some $B \in [n+1]^{k+1}$. The \textit{face restriction map}
  $$
    \Res^k_B : \check{C}^k(\vec\kappa) \to \check{C}^k(\vec\kappa\res_B)
  $$
  is defined as follows: for $\vec x \in \prod_{i \in B} \kappa_i,$ extend $\vec x$ to $\vec x^\infty$ by
  $$
    x^\infty_i = 
    \begin{cases}
      x_i,  & i  \in B, \\
      \infty, & i \not\in B.
    \end{cases}
  $$
  Then, given a tuple $\vec f = \langle f_A : A \in [n+1]^{k+1} \rangle \in \check{C}^k(\vec\kappa)$, $\Res^k_B\vec f$ is the single map $f_B^\infty$ which sends $\vec x$ to $f_B(\vec x^\infty)$.
\end{definition}

\begin{proposition}\label{prop:res}
  For $0 < k < n$ and $B \in [n+1]^{k+1}$, $\Res^k_B$ is well-defined on cohomology.
\end{proposition}

\begin{proof}
  The assertion that $\Res^k_B$ maps cocycles to cocycles is nearly vacuous, since top cocycles are just ordinary set functions. To see that coboundaries are mapped to coboundaries, suppose $\vec f = d\vec g$ for some $\vec g \in \check{C}^{k-1}(\vec\kappa)$. For each $j \in B$ define $g_{B \setminus \{j\}}^\infty$ exactly as for $\vec f$, so that $\vec x \in \prod_{i \in B} \kappa_i$ is mapped to $g_{B \setminus \{j\}}(\vec x^\infty)$. Now each $g_{B\setminus \{j\}}^\infty \in \calC_{B \setminus \{j\}}(\vec\kappa\res_B)$, and $\sum_{j \in B} g_{B\setminus \{j\}}^\infty = f_B^\infty = \Res^k_B(\vec f)$.
\end{proof}

As an immediate corollary, we note that if $\Res^k_B(\vec f)$ is nonzero in the top cohomology of $X(\vec\kappa\res_B)^-$, then $\vec f$ must be nonzero in $H^k(\xkappa^-,\underline{\bbZ_2})$. This is the manner in which we use the face restriction map to reduce questions on intermediate cohomology to questions of top cohomology.

We next examine the cocycle condition more closely.

\begin{definition}\label{def:phi}
  Fix $\vec\kappa = \langle \kappa_0, \ldots, \kappa_n\rangle, B \in [n+1]^{k+1}$, and let $\vec f \in H^k(\xkappa^-,\underline{\bbZ_2})$ be a $k$-cocycle. For a tuple 
  $$
    \vec x = \langle x_i \in \kappa_i \cup \{\infty\} : i \not \in B\rangle \in X(\vec\kappa^B),
  $$ 
  and a choice of remaining coordinates $\vec y \in \prod_{i \in B} \kappa_i$, we set
  $$
    \phi_{\vec f,B,\vec{x}}(\vec y) = f_B(\vec x, \vec y).
  $$
  The $\phi_{\vec f, B, \vec x}$ then assemble into a function
  $$
    \vec x \mapsto \phi_{\vec f, B, \vec x} : X(\vec\kappa^B) \to Z^k(\vec\kappa\res_B)
  $$
  which we denote by $\phi_{\vec f, B}$.
\end{definition}

Note that the codomain of $\phi_{\vec f, B}$ is naturally equipped with the product topology (i.e., that of pointwise convergence).

\begin{proposition}\label{prop:ptwise-limit}
  For any $k$-cocycle $\vec f$, $B$, the map $\phi_{\vec f, B}$ is continuous. Moreover, for any $\vec x \in \prod_{i \not \in B} \alpha(\kappa_i)$ except possibly $(\infty,\ldots,\infty)$, $\phi_{\vec f, B}(\vec x)$ is a coboundary.
\end{proposition}

\begin{proof}
  First notice that $Z^k(\vec\kappa\res_B) = \{f : \prod_{i \in B} \kappa_i \to \bbZ_2\}$ is homeomorphic to a product of one copy of $\bbZ_2$ for each $\vec y \in \prod_{i \in B} \kappa_i$; thus to show $\phi_{\vec f, B}$ is continuous, it suffices to show each of the coordinate functions
  \begin{align*}
    \phi_{\vec f, B}^{\vec y} \quad&:\quad X(\vec\kappa^B) \to \bbZ_2 \\
    \vec x \in \prod_{i \not \in B} \alpha(\kappa_i) \quad&\mapsto\quad f_B(\vec x,\vec y)
  \end{align*}
  is continuous. This follows directly from the continuity of $f_B$ on $\calD_B$, since the domain of $\phi^{\vec y}_{\vec f, B}$ is open in $\calD_B$.

  Now, let $\vec x \in \prod_{i \not \in B} \alpha(\kappa_i)$ be a point other than $(\infty,\ldots,\infty)$, so there exists $j \not \in B$ such that $x_j \neq \infty$. Set $A = B \cup \{j\}$, and note that the cocycle condition for $\vec f$ guarantees that
  $$
    f_B = \sum_{i \in B} f_{B \setminus \{i\} \cup \{j\}}.
  $$
  So we obtain
  $$
    \phi_{\vec f, B}(\vec x) = f_B(\vec x,-) = \sum_{i \in B} f_{B \setminus \{i\} \cup \{j\}}(\vec x,-)
  $$
  and the latter expresses $\phi_{\vec f, B}(\vec x)$ as a sum of functions, each mod-finite constant in one of the coordinates $i \in B$.
\end{proof}

Proposition \ref{prop:ptwise-limit} can be construed as stating that a cocycle $\vec f$ can be written as a certain pointwise limit of coboundaries. For Theorem \ref{thm:main3}, we are interested in the special case where a single map $\psi : X(\vec\kappa^B)^- \to \check{C}^{k-1}(\vec\kappa\res_B)$ induces a family of coboundaries via $\vec x \mapsto d\psi(\vec x)$. In that context, we have the following partial converse.

\begin{proposition}\label{prop:ptwise-limit-converse}
  Let $\psi : X(\vec\kappa^B)^- \to \check{C}^{k-1}(\vec\kappa\res_B)$ be a continuous map such that $\vec x \mapsto d\psi(\vec x): X(\vec\kappa^B)^- \to Z^k(\vec\kappa\res_B)$ extends to a continuous map $\phi$ on $X(\vec\kappa^B)$. Then there exists a $k$-cocycle $\vec f \in H^k(\xkappa^-, \underline{\bbZ_2})$ such that $\phi(\infty,\ldots,\infty) = \Res^k_B(\vec f)$.
\end{proposition}

Here, the topology on $\check{C}^{k-1}(\vec\kappa\res_B)$ is the product of the topologies on the $\calC_A$, which in turn are given the compact-open topology.

\begin{proof}
  Let $A \in [n+1]^{k+1}, \vec x \in \calD_A$. To define $f_A$, we distinguish three cases:
  \begin{itemize}
    \item If $A = B$, we set $f_A(\vec x) = \phi(\vec x^B)(\vec x\res_B)$.
    \item If $A = B \setminus \{s\} \cup \{t\}$ we set $f_A(\vec x) = \psi(\vec x^B)_{B \setminus \{s\}}(\vec x\res_B)$. Note this is well-defined since $x_t \neq \infty$, and independent of $t$.
    \item In all other cases we set $f_A$ to be identically zero.
  \end{itemize}
  We first show that the $f_A$ so defined form a cochain, i.e., are each continuous. If $A = B$, then continuity of $f_A$ is exactly continuity of $\phi$, and if $|A \triangle B| > 2$ then $f_A$ is zero, hence certainly continuous. So suppose $A = B \setminus \{s\} \cup \{t\}, s \in B, t \not \in B$; in this case $f_A(\vec x) = \psi(\vec x^B)_{B \setminus \{s\}}(\vec x\res_B)$. We can write $f_A$ as the composition of $\psi \times \id$ with the evaluation map $\calC_{B \setminus \{s\}} \times \calD_{B \setminus \{s\}} \to \bbZ_2$, both of which are continuous; for the latter map this holds since the $\calD$-spaces are all locally compact Hausdorff. Thus in all cases $\vec f \in \check{C}^k(\xkappa^-,\underline{\bbZ_2})$.

  To see that $\vec f$ is moreover a $k$-cocycle, let $A' \in [n+1]^{k+2}$. First suppose $B \subset A',$ so that $A' = B \cup \{t\}$ for some $t \not \in B$. We calculate, for any $\vec x \in \calD_{A'}(\vec\kappa)$,
  \begin{align*}
    \sum_{i \in A'} f_{A' \setminus \{i\}}(\vec x) &= f_B(\vec x) + \sum_{s \in B} f_{B \setminus \{s\} \cup \{t\}}(\vec x) \\
    &= f_B(\vec x) + \sum_{s \in B} \psi(\vec x^B)_{B \setminus \{s\}}(\vec x\res_B) \\
    &= \phi(\vec x^B)(\vec x\res_B) + d\psi(\vec x^B)(\vec x\res_B) \\
    &= 0,
  \end{align*}
  where the last equality holds since $x_t \neq \infty$ so $\vec x^B \in X(\vec\kappa^B)^-$. If instead $A'$ contains exactly two elements not in $B$, write $A' = B \setminus \{s_0\} \cup \{t_0,t_1\}$ where $s_0 \in B$ and $t_0,t_1 \not \in B$ are distinct. Then we calculate
  \begin{align*}
    \sum_{i \in A'} f_{A' \setminus \{i\}}(\vec x) &= f_{B \setminus \{s_0\} \cup \{t_0\}}(\vec x) + f_{B \setminus \{s_0\} \cup \{t_1\}}(\vec x) \\
    &+ \sum_{s_1 \in B \setminus \{s_0\}} f_{B \setminus \{s_0,s_1\} \cup \{t_0,t_1\}}(\vec x).
  \end{align*}
  The first two summands are equal due to the independence of the definition from $t_0,t_1$; in the latter sum all functions are identically zero. So the overall sum is zero.

  Lastly, if $A'$ contains at least three elements not in $B$, then each $A' \setminus \{i\}$ contains at least two elements not in $B$, so that each term in the sum is zero. This concludes the proof that $\vec f$ is a cocycle.

  We finish by observing that substituting $(\infty,\ldots,\infty)$ in the definition of $\phi_{\vec f, B}$ yields a map sending $\vec x$ to $f_B(\vec x^\infty)$, which is exactly $\Res^k_B(\vec f)$. 
\end{proof}

Theorem \ref{thm:main3} now follows from the results proved above: given $\psi$ as in the statement, Proposition \ref{prop:ptwise-limit-converse} gives $\vec f$ a cocycle which cannot be a coboundary by Proposition \ref{prop:res}.

To finish this section, we apply Theorem \ref{thm:main3} to settle a case not handled by the methods of sections \ref{sec:main1} or \ref{sec:main2}:

\begin{proposition}\label{prop:main3-app}
  Let $z, k$ be positive integers, and let $\vec\kappa$ be a tuple of length $z+k+1$ such that
  \begin{itemize}
    \item $\kappa_i = \aleph_0$ for $i < z$, and
    \item for $i \geq z$ the $\kappa_i$ are uncountable and $\kappa_i < \cf{\kappa_{i+1}}$ for $z \leq i < z+k$.
  \end{itemize} 
  Then $H^k(\xkappa^-,\underline{\bbZ_2}) \neq 0$.
\end{proposition} 

Note this shows $H^n(X(\aleph_0,\aleph_1,\ldots,\aleph_{n+1})^-,\underline{\bbZ_2}) \neq 0$ in ZFC.

\begin{proof}
  To apply Theorem \ref{thm:main3}, we let $B = \{z,z+1,\ldots,z+k\}$, so that
  $X(\vec\kappa^B) = \alpha(\aleph_0)^z$, the cartesian product of $z$ many copies of the one-point compactification of $\aleph_0$. We define $\psi : X(\vec\kappa^B)^- \to \check{C}^{k-1}(\vec\kappa\res_B)$ by
  $$
    \psi(m_0,\ldots, m_{z-1})_{B \setminus \{z\}}(\beta_0,\ldots,\beta_k) = 
    \begin{cases}
      1, &\t{ if $\beta_0 < \min(m_0,\ldots,m_{z-1})$}\\
         &\t{ \quad and all other $\beta_i$ are even,} \\
      0, &\t{ otherwise.}
    \end{cases}
  $$
  and all other components zero.
  (Here, by ``even'' we mean ``of the form $\lambda + 2k$ for $\lambda$ a limit ordinal, $k < \omega$''. This can be replaced with any sequence of unbounded, co-unbounded subsets of the $\kappa_i$.)

  We make a few observations.
  \begin{itemize}
    \item The codomain of $\psi$ is indeed $\check{C}^{k-1}(\vec\kappa\res_B)$. To see this, note that if not all $m_i$ are $\infty$, the resulting map is mod-finite constant in the $\beta_0$-direction, and thus lies in $\calC_{B \setminus \{z\}}(\vec\kappa \res_B).$
    \item $\psi$ is locally constant, and thus continuous: if any of the $m_i$ are not $\infty$ then $\psi$ is constant on the neighborhood consisting of all those tuples agreeing with $\vec m$ on its non-$\infty$ coordinates and greater than $\min \vec m$ on the others.
    \item And $\vec m \mapsto d\psi(\vec m)$, which is just the $B \setminus \{z\}$ component, extends continuously to a map $\phi$ on all of $X(\vec\kappa^B)$, defined by the same formula. This is continuous at $(\infty,\ldots,\infty)$ since for any $\vec \beta$, the set of $\vec m$ for which $\beta_0 < \min\vec m$ is clopen.
  \end{itemize}
  
So, modulo Theorem \ref{thm:main3}, we need only show that the map
  \begin{align*}
    \phi(\infty,\ldots,\infty) &:\quad \left(\prod_{z \leq i \leq z+k} \kappa_i\right) \to \bbZ_2\\
    \langle \beta_0,\ldots,\beta_k \rangle &\mapsto
    \begin{cases}
      1, &\t{ if $\beta_0 < \aleph_0$} \\
         &\quad\t{ and all other $\beta_i$ are even,} \\
      0  &\t{ otherwise}
    \end{cases}
  \end{align*}
  is nontrivial as a top cocycle in $\vec \kappa$.

  To do this, we define properties $P_j$ for $j \leq k$ which may or may not hold of functions $f : \prod_{i \leq j} \kappa_{z+i} \to \bbZ_2$, inductively as follows:
  
  \begin{itemize}
    \item A function $f : \kappa_z \to \bbZ_2$ has property $P_0$ if and only if it's constant on a cofinite subset of $\kappa_z$.
    \item For $j > 0$, a function $f : \prod_{i \leq j} \kappa_{z+i} \to \bbZ_2$ has property $P_j$ if and only if there is $\gamma^* \in \kappa_{z+j}$ such that for any $\gamma_0,\gamma_1 \geq \gamma^*, f(-,\gamma_0) + f(-,\gamma_1)$ has property $P_{j-1}$.
  \end{itemize}
\
  We first show that for $j \geq 1$, if $f : \prod_{i \leq j} \kappa_{z+i} \to \bbZ_2$ is a coboundary, then $f$ has property $P_j$:
  \begin{itemize}
    \item If $j = 1$, then $f : \kappa_z \times \kappa_{z+1} \to \bbZ_2$ is such that there are $g_z,g_{z+1}$ such that $g_z + g_{z+1} = f$ and $g_{z+i}$ is mod-finite constant in the $i$-th coordinate. Then for every $\beta \in \kappa_z,$ the set
    $$
      \{\gamma \in \kappa_{z+1} : g_{z+1}(\beta,\gamma) \neq g_{z+1}(\beta,\infty)\}
      \subset \kappa_{z+1}
    $$
    is finite; since $\kappa_z < \cf\kappa_{z+1}$, we can find $\gamma^* \in \kappa_{z+1}$ above all these finite sets. Put $\gamma_0,\gamma_1 \geq \gamma^*$ and take $B \subset \kappa_z$ cofinite such that, for any $\beta \in B$ we have
    $$
      g_z(\beta,\gamma_0) = g_z(\infty,\gamma_0) \t{ and } g_z(\beta,\gamma_1) = g_z(\infty,\gamma_1).
    $$
    Then for any $\beta \in B$ we have
    \begin{align*}
      &\phantom{=i}f(\beta,\gamma_0) + f(\beta,\gamma_1) \\
      &= g_{z}(\beta,\gamma_0) + g_{z+1}(\beta,\gamma_0) + g_z(\beta,\gamma_1) + g_{z+1}(\beta,\gamma_1) \\
      &= g_{z}(\infty,\gamma_0) + g_{z+1}(\beta,\infty) + g_z(\infty,\gamma_1) + g_{z+1}(\beta,\infty) \\
      &= g_z(\infty,\gamma_0) + g_z(\infty,\gamma_1)
    \end{align*}
    which is independent of $\beta$. So $f(-,\gamma_0) + f(-,\gamma_1)$ is constant on a cofinite set and thus has property $P_1$.
    \item If $j > 1$, suppose $f : \prod_{i \leq j} \kappa_{z+i} \to \bbZ_2$ is a coboundary, so $f = \sum_{i \leq j} g_{z+i}$ where $g_{z+i}$ is mod-finite constant in the $i$-th coordinate. In this case, for every $\vec \beta \in \prod_{i < j} \kappa_{z+i}$, the set
    $$
      \{\gamma \in \kappa_{z+j} : g_{z+j}(\vec\beta,\gamma) \neq g_{z+j}(\vec\beta,\infty)\}
    $$
    is again finite, so since $\kappa_{z+j-1} < \cf\kappa_{z+j}$ we can find $\gamma^* \in \kappa_{z+j}$ such that $\gamma \geq \gamma^*$ implies $g_{z+j}(-,\gamma) = g_{z+j}(-,\infty)$. Let $\gamma_0,\gamma_1 \geq \gamma^*$. Then we get
    \begin{align*}
      &\phantom{=i} f(-,\gamma_0) + f(-,\gamma_1) \\
      &= \left(\sum_{i < j} g_{z+i}(-,\gamma_0)\right) + g_{z+j}(-,\gamma_0)
       + \left(\sum_{i < j} g_{z+i}(-,\gamma_1)\right) + g_{z+j}(-,\gamma_1) \\
      &= \sum_{i < j}\left( g_{z+i}(-,\gamma_0) + g_{z+i}(-,\gamma_1)\right),
    \end{align*}
    which exhibits $f(-,\gamma_0) + f(-,\gamma_1)$ as a coboundary. By induction this means $f(-,\gamma_0) + f(-,\gamma_1)$ has property $P_{j-1}$ so that $f$ has property $P_j$.
  \end{itemize}
  It thus remains to show that our map $\phi(\infty,\ldots,\infty)$ does not have property $P_k$. We again do this by induction on $k$:
  \begin{itemize}
    \item If $k = 1$ then $\phi(\infty,\ldots,\infty)(\beta,\gamma) = 1$ if and only if $\beta < \aleph_0$ and $\gamma$ is even. Fixing any $\gamma^*$, take $\gamma_0,\gamma_1 > \gamma^*$ even and odd, respectively. Then $\phi(\infty,\ldots,\infty)(\beta,\gamma_0) + \phi(\infty,\ldots,\infty)(\beta,\gamma_1) = 1$ if and only if $\beta < \aleph_0$; since $\kappa_z > \aleph_0$, this is not constant on a cofinite subset of $\kappa_z$.
    \
    \item Now, suppose $\phi(\infty,\ldots,\infty)$ had property $P_j$, with witness $\gamma^* \in \kappa_{z+j}$. Again choose $\gamma_0,\gamma_1 \geq \gamma^*$ even and odd respectively, and notice that 
    $$
      \phi(\infty,\ldots,\infty)(\vec \beta,\gamma_0) + \phi(\infty,\ldots,\infty)(\vec \beta,\gamma_1) = \phi(\infty,\ldots,\infty)(\vec\beta).
    $$
    The left hand side has property $P_{j-1}$ by hypothesis, but the right side does not, by induction. Thus $\phi(\infty,\ldots,\infty)$ does not have property $P_j$.
  \end{itemize}
  We conclude by Theorem \ref{thm:main3} that $H^k(\xkappa^-,\underline{\bbZ_2}) \neq 0$.
\end{proof}

Proposition \ref{prop:main3-app} is typical of the intended application of Theorem \ref{thm:main3}: namely, to reduce an intermediate cohomology computation to a question in top cohomology, which is more easily tractable with elementary inductions like those in the proof.

\section{Further work and conclusion}
Chief among the outstanding questions is whether the cardinal arithmetic assumptions of section \ref{sec:main1} are necessary, and indeed whether any other independence phenomena arise:

\begin{question}
  Does ZFC decide the vanishing/nonvanishing of $H^k(\xkappa^-,\underline{\bbZ_2})$ for all $\vec\kappa$?
\end{question}

Additional questions about the $X(\vec\kappa)^-$ are important in applications as well. Aoki shows that the cup product structure on the cohomology ring of 
$$
  X(\aleph_0,\aleph_1)^- \times \cdots \times X(\aleph_{2n,2n+1})^-
$$
admits a nonzero product of size $n+1$, and asks if there is a smaller locally profinite set with this property. Further work is required to determine if the machinery introduced here can be applied to resolve this question.

Lastly, as we have not resolved all cohomology groups of the $\xkappa^-$, it is worth pointing out some small cases that remain unsolved:
\begin{question}
  Which of the following are consistent?
  \begin{itemize}
    \item $H^1(X(\aleph_1,\aleph_2,\aleph_2)^-,\underline{\bbZ_2}) = 0$? $H^1(X(\aleph_1,\aleph_2,\aleph_2)^-,\underline{\bbZ_2}) \neq 0$?
    \item $H^1(X(\aleph_2,\aleph_2,\aleph_2)^-,\underline{\bbZ_2}) = 0$? $H^1(X(\aleph_2,\aleph_2,\aleph_2)^-,\underline{\bbZ_2}) \neq 0$?
    \item $H^1(X(\aleph_1,\aleph_2,\aleph_3)^-,\underline{\bbZ_2}) = 0$? (consistently nonzero if $2^{\aleph_1} \geq \aleph_3$)
  \end{itemize}
\end{question}

\bibliography{locprof-main}
 
\end{document}